\documentclass[a4paper,reqno]{amsart}
\usepackage{amssymb}
\usepackage[utf8]{inputenc}
\usepackage[T1]{fontenc}
\usepackage{ae,icomma}
\usepackage{enumerate}
\usepackage[b]{esvect}
\usepackage[pagebackref,colorlinks,linkcolor=red,citecolor=blue,urlcolor=blue,hypertexnames=true]{hyperref}
\usepackage{bbm} % for the complex imaginary unit i
\usepackage{ushort}

\usepackage{mathrsfs}

\newtheorem{theorem}{Theorem}[subsection]
\newtheorem{lemma}[theorem]{Lemma}
\newtheorem{lem}[theorem]{Lemma}
\newtheorem{proposition}[theorem]{Proposition}
\newtheorem{prop}[theorem]{Proposition}
\newtheorem{cor}[theorem]{Corollary}
\newtheorem{thm}[theorem]{Theorem}
\newtheorem{corollary}[theorem]{Corollary}

\theoremstyle{definition}
\newtheorem{definition}[theorem]{Definition}

\newtheorem{remark}[theorem]{Remark}
\newtheorem{rem}[theorem]{Remark}
\newtheorem{example}[theorem]{Example}

\newcommand\N{\mathbb N}
\newcommand\Z{\mathbb Z}
\newcommand\Q{\mathbb Q}
\newcommand\R{\mathbb R}

\newcommand\ph\varphi
\newcommand\ps\psi
\newcommand\ep\varepsilon
\newcommand\rh\varrho
\newcommand\al\alpha
\newcommand\be\beta
\newcommand\ga\gamma
\newcommand\om\omega
\newcommand\ta\tau
\renewcommand\th\vartheta
\newcommand\de\delta
\newcommand\ze\zeta
\newcommand\ch\chi
\newcommand\et\eta
\newcommand\io\iota
\newcommand\la\lambda
\newcommand\si\sigma

\newcommand\Ga\Gamma
\newcommand\De\Delta
\newcommand\Th\Theta
\newcommand\La\Lambda
\newcommand\Si\Sigma
\newcommand\Ph\Phi
\newcommand\Ps\Psi
\newcommand\Om\Omega

\newcommand\x{\ushort X}
\newcommand\ac{[\x]}
\newcommand\cx{[\x]}
\newcommand\rx{\R\ac}
\newcommand\RR{\R}
\newcommand\rxlin{\rx_1}
\newcommand\sos{\sum\rx^2}
\newcommand\nn{_{\ge0}}
\newcommand\pos{_{>0}}
\newcommand\m{\R^{m\times m}}

\newcommand\sm{S\R^{m\times m}}
\newcommand\smpd{S\R^{m\times m}_{\succ0}}
\newcommand\smpsd{S\R^{m\times m}_{\succeq0}}
\renewcommand\emptyset\varnothing
\newcommand\vecx[1]{{\footnotesize\vv{\ac}_#1}}
\newcommand\rad[1]{\sqrt{#1}}
\newcommand\rrad[1]{\sqrt[\displaystyle r]{#1}}
\newcommand\s[1]{s(#1)}

\DeclareMathOperator\Span{span}
\DeclareMathOperator\rank{rank}

\DeclareMathOperator\im{im}
\DeclareMathOperator\ev{ev}
\DeclareMathOperator\dist{dist}
\DeclareMathOperator\supp{supp}
\DeclareMathOperator\tr{tr}
\renewcommand{\setminus}{\smallsetminus}

\newcommand{\eop}{\hfill$\Box$}
\def\SRdd{S\R^{d \times d} }

\def\Rmm{\R^{m \times m} }

\def\cD{ {\mathfrak S} }
\def\bes{\begin{equation*}}
\def\ees{\end{equation*}}

\def\beq{\begin{equation} }
\def\eeq{\end{equation} }

\def\ben{\begin{enumerate} }
\def\een{\end{enumerate} }
\DeclareFontFamily{OT1}{manual}{}
\DeclareFontShape{OT1}{manual}{m}{n}{ <10> manfnt }{}

\title[Infeasibility certificates for linear matrix inequalities]{Infeasibility certificates for\\linear matrix inequalities}

\author[Igor Klep]{Igor Klep}
\address{Igor Klep, Univerza v Mariboru, Fakulteta za naravoslovje in matematiko,
Koro\v ska 160, SI--2000 Maribor, Slovenia, and Univerza v Ljubljani, Fakulteta za matematiko
in fiziko, Jadranska 21, SI--1111 Ljubljana, Slovenia}
\email{igor.klep@fmf.uni-lj.si}
\author[Markus Schweighofer]{Markus Schweighofer}
\address{Markus Schweighofer,
Universit\"at Konstanz,
Fachbereich Mathematik und Statistik,
78457 Konstanz, Allemagne}
\email{markus.schweighofer@uni-konstanz.de}
\date{25 August 2011}
\thanks{This research was supported through the programme ``Research in Pairs'' (RiP) by the Mathematisches Forschungsinstitut Oberwolfach in 2010.
The first author was also supported by the Slovenian Research Agency under Project no. J1-3608 and Program no. P1-0222. Part of the research
was done while the first author held a visiting professorship at the Universität Konstanz in 2011}

\begin{document}
\begin{abstract}
Farkas' lemma is a fundamental result from \emph{linear programming} providing \emph{linear} certificates for infeasibility of systems of linear inequalities.
In \emph{semidefinite programming}, such linear certificates only exist for \emph{strongly} infeasible linear matrix inequalities. 
We provide \emph{nonlinear algebraic} certificates for \emph{all} infeasible linear matrix inequalities in the spirit of real algebraic geometry.
More precisely, we show that a linear matrix inequality $L(x)\succeq 0$ is infeasible if and only if $-1$ lies in the quadratic module associated to $L$.
We prove exponential degree bounds for the corresponding algebraic certificate. In order to get a polynomial size certificate, we use a more
involved algebraic certificate motivated by the real radical and Prestel's theory of semiorderings.
Completely different methods, namely complete positivity from operator algebras, are employed to consider linear matrix inequality domination.
\end{abstract}

\maketitle

\noindent
A \emph{linear matrix inequality} (LMI) is a condition of the form
\[
L(x) = A_0 + \sum_{i=1}^n x_i A_i \succeq0\qquad(x\in\R^n)
\]
where the $A_i$ are symmetric matrices of the same size and one is interested in the solutions $x\in\R^n$ making $L(x)\succeq0$, i.e., making
$L(x)$ into a positive semidefinite matrix.
The solution set to such an inequality is a closed convex semialgebraic subset of $\R^n$ called a {\em spectrahedron}. Optimization of linear
objective functions over spectrahedra is called \emph{semidefinite programming} (SDP) \cite{bv,to,wsv}. In this article, we are mainly concerned with the important
SDP feasibility problem: When is an LMI \emph{feasible}; i.e., when is there
an $x\in\R^n$ such that $L(x)\succeq0$?

Note that a diagonal LMI (all $A_i$ diagonal matrices) is just a (finite) system of (non-strict) linear inequalities. The solution set of such a linear
system is a (closed convex) polyhedron. Optimization of linear
objective functions over polyhedra is called \emph{linear programming} (LP).
The ellipsoid method developed by Shor, Yudin, Nemirovskii and Khachiyan
showed at the end of the 1970s for the first time that the LP feasibility problem (and actually the problem of solving LPs) can be solved in
(deterministically) polynomial time (in the bit model of computation assuming rational coefficients) \cite[Chapter 13]{sr}. Another breakthrough
came in the 1980s with the introduction of the more practical interior point methods by Karmarkar and their theoretical underpinning by
Nesterov and Nemirovskii \cite{nn,nem}.

The motivation to replace the prefix ``poly'' by ``spectra'' is to replace ``many'' values of linear polynomials (the diagonal values of $L(x)$) by  the ``spectrum'' of $L(x)$ (i.e., the set of its eigenvalues, all of which are real since $L(x)$ is symmetric). 
The advantage of LMIs over systems of linear inequalities (or of spectrahedra over polyhedra, and SDP over LP, respectively) is a considerable gain of
expressiveness which makes LMIs an important tool in many areas of applied and pure mathematics.
Many problems in control theory, system identification and signal processing can be formulated using
LMIs \cite{befb,par,sig}. Combinatorial optimization problems can often be modeled or approximated by SDPs \cite{go}.
LMIs also find application in real algebraic geometry for finding sums of squares decompositions of polynomials \cite{las,ma}.
Strongly related to this, there is a hierarchy of SDP approximations to polynomial optimization problems \cite{lau} consisting of the so-called
Lasserre moment relaxations. This hierarchy and related methods recast the field of polynomial optimization (where the word ``polynomial''
stands for polynomial objective and polynomial constraints). In this article, rather than trying to solve polynomial optimization problems by
using SDPs, we borrow ideas and techniques from real algebraic geometry and polynomial optimization in order to get new results
in the theory of semidefinite programming.

The price to pay for the increased expressivity of SDPs is that they enjoy some less good properties. The complexity of solving general
SDPs is a very subtle issue which is often downplayed. For applications in combinatorial optimization, it follows typically from the general theory of the
ellipsoid method \cite{sr} or interior point methods \cite{nn} that the translation into SDPs yields a polynomial time algorithm (see for instance
\cite[Section 1.9]{dk} for exact statements). However the complexity status of the LMI feasibility problem (the problem of deciding
whether a given LMI with rational coefficients has a solution) is largely unknown. What is known is essentially
only that (in the bit model) LMI feasibility lies
either in $\text{NP}\cap\text{co-NP}$ or outside $\text{NP}\cup\text{co-NP}$. This will also follow from our work below, but has been already proved by
Ramana \cite{ra} in 1997. The standard (Lagrange-Slater) dual of a semidefinite program works well when the feasible set is full-dimensional
(e.g. if there is $x\in\R^n$ with $L(x)\succ0$). However, in general strong duality can fail badly which is a serious problem since there is no easy
way of reducing to the full-dimensional case (cf.~the LMI feasibility problem mentioned above). Even the corresponding version of Farkas' lemma
fails for SDP.

Ramana's as well as our proof relies on a extension of the standard (Lagrange-Slater) dual of an SDP which can be produced
in polynomial time from the primal and for which strong duality (more precisely zero gap and dual attainment) always holds.
Ramana's extension is an encoding of the standard dual of a regularized primal SDP in the sense of Borwein and Wolkowicz
\cite{bw}. Our dual, although having some superficial similarity to Ramana's, relies on completely different ideas, namely
sums of squares certificates of nonnegativity. The ideas for this \emph{sums of squares dual} come from real algebraic geometry,
more precisely from sums of squares
representations and the Real Nullstellensatz \cite{ma, pd, sc}. We 
believe that this new connection will lead to further insights in the future.

The paper is organized as follows:
We fix terminology and notation in Section \ref{sec:not}.
Our first line of results is given in Section \ref{sec:qm},
where we give an algebraic characterization of infeasible
LMIs (see Theorem \ref{thm:main1} and Corollary \ref{cor:bounds}) involving 
the quadratic module associated to an LMI. Our characterization allows us to construct a new
LMI whose feasibility is equivalent to the infeasibility of the
original LMI. This new LMI is canonical from the viewpoint of positive polynomials.
However its size is exponential in the size of the primal, so 
in Section \ref{sec:ram} we use real algebraic geometry 
to  construct a polynomial
size LMI whose feasibility is equivalent to the infeasibility
of the original LMI, cf.~Theorem \ref{sosram}.
At the same time Theorem \ref{sosram} gives a new type of a 
linear Positivstellensatz characterizing linear polynomials nonnegative 
on a spectrahedron.
The article concludes with Section \ref{sec:hkm}, 
where we revisit the \cite{hkm} noncommutative (matricial) relaxation of 
an LMI and investigate how 
our duality theory of SDP pertains to (completely) positive maps.

\section{Notation and terminology}\label{sec:not}

We write $\N:=\{1,2,\dots\}$, $\Q$, and $\R$ for the sets of natural, rational, and real numbers, respectively. For any matrix $A$,
we denote by $A^*$ its transpose.

\subsection{Sums of squares}

Let $R$ be a (commutative unital) ring.
Then $SR^{m\times m}:=
\{A\in R^{m\times m}\mid A=A^*\}$ denotes the set of all \emph{symmetric} $m\times m$ matrices.
Examples of these include \emph{hermitian squares}, i.e.,
elements of the form $A^*A$ for some $A\in R^{m\times m}$.

Recall that a matrix $A\in\Rmm$ is called
\textit{positive semidefinite} (\emph{positive definite}) if it is
symmetric and $v^*Av\ge 0$ for all (column) vectors $v\in\R^m$, 
$A$ is \textit{positive definite} if it is positive semidefinite and invertible.
For real matrices $A$ and $B$ of the same size, we write $A\preceq B$
(respectively $A\prec B$) to express that $B-A$ is positive semidefinite
(respectively positive definite). We denote by $S\R^{m\times m}_{\succeq0}$ and $S\R^{m\times m}_{\succ0}$ the convex cone
of all positive semidefinite and positive definite
matrices of size $m$, respectively.

Let $\x=(X_1,\dots,X_n)$ be an $n$-tuple of 
commuting variables and $\R\cx$ the polynomial ring.
With $\R\cx_k$ we denote the vector space of all polynomials
of degree $\leq k$, and $\sos$ is the convex cone of all sums of squares (\emph{sos-polynomials}),
i.e.,
$$
\sos=\Big\{ \sum_{i=1}^r p_i^2 \mid r\in\N,\, p_i\in\R\cx\Big\}.
$$
A (real) \emph{matrix polynomial} is a matrix whose entries are polynomials from $\RR\cx$. It is \emph{linear} or \emph{quadratic} if its entries are from $\RR\cx_1$ or
$\RR\cx_2$, respectively. A matrix polynomial is called \emph{symmetric}
if it coincides with its transpose. An example of symmetric matrix polynomials that are of special interest to us are sums of hermitian
squares in $\RR\cx^{m\times m}$. They are called \emph{sos-matrices}. More explicitly,
$S\in\RR\cx^{m\times m}$ is an sos-matrix if the following equivalent
conditions hold:
\begin{enumerate}[\rm(i)]
\item $S=P^*P$ for some $s\in\N$ and some $P\in\RR\cx^{s\times m}$;
\item $S=\sum_{i=1}^rQ_i^*Q_i$ for some $r\in\N$ and $Q_i\in\RR\cx^{m\times m}$;
\item $S=\sum_{i=1}^sv_iv_i^*$ for some $s\in\N$ and $v_i\in\RR\cx^m$.
\end{enumerate}
Note that an sos-matrix $S\in\RR\cx^{m\times m}$ is positive semidefinite on 
$\R^n$
but not vice-versa, since e.g.~a polynomial nonnegative on $\R^n$ is not necessarily a sum of squares of polynomials \cite{ma,pd}.

For a comprehensive treatment of the theory of matrix polynomials we refer the reader to the book \cite{glr} and the references therein.

\subsection{Linear pencils, spectrahedra, and quadratic modules}

We use the term \emph{linear pencil} as a synonym and abbreviation for symmetric linear matrix polynomial.

Let $R$ be a (commutative unital) ring. We recall that in real algebraic geometry
a subset $M\subseteq R$ is called a \emph{quadratic module} in $R$ if it
contains $1$ and is closed under addition and multiplication with squares, i.e.,
$$1\in M,\quad M+M\subseteq M\quad\text{and}\quad a^2M\subseteq M\text{\ for all $a\in R$},$$
see for example \cite{ma}. A quadratic module $M\subseteq R$ is called \emph{proper} if $-1\not\in M$. Note that an improper
quadratic module $M$ in a ring $R$ with $\frac12\in R$ equals $R$ by the identity
\begin{equation}\label{id4}
4a=(a+1)^2-(a-1)^2\qquad\text{for all $a\in R$}.
\end{equation}
An LMI $L(x)\succeq0$ can be seen as the infinite family of simultaneous linear inequalities $u^*L(x)u\ge0$ ($u\in\R^m$).
In optimization,
when dealing with families of linear inequalities, one 
often considers the convex cone generated by them
(cf.~$C_L$ in the definition below). Real algebraic geometry handles 
\emph{arbitrary} polynomial inequalities and
uses the multiplicative structure of the polynomial ring. Thence
one considers more special types of convex cones like in our case
quadratic modules ($M_L$ from the next definition). One of the aims of this article is to show that it is advantageous to consider
quadratic modules for the study of LMIs. Since quadratic
modules are infinite-dimensional convex cones, we will later on also consider certain finite-dimensional truncations of them, see Subsection
\ref{subs:degbound}.

\begin{definition}\label{def:scm}
Let $L$ be a linear pencil of size $m$ in the variables $\x$. We introduce
\begin{align*}
S_L&:=\{x\in\R^n\mid L(x)\succeq0\}\subseteq\R^n,\\
C_L&:=\Big\{c+\sum_iu_i^*Lu_i\mid c\in\R\nn,\ u_i\in\R^m\Big\}\\
&~=\Big\{c+\tr(LS)\mid c\in\R\nn,\ S\in S\R^{m\times m}_{\succeq0}\Big\}
\subseteq\R\ac_1,\\
M_L&:=\Big\{s+\sum_iv_i^*Lv_i\mid s\in\sos,\ v_i\in\rx^m\Big\}\\
&~=\Big\{s+\tr(LS)\mid s\in\R\cx\text{ sos-polynomial},\ S\in\R\cx^{m\times m}\text{ sos-matrix}\Big\}
\subseteq\R\ac,
\end{align*}
and call them the \emph{spectrahedron} (or LMI set), the \emph{convex cone}, and the \emph{quadratic module}
associated to the linear pencil $L$, respectively. Call the \emph{linear matrix inequality} (LMI)
$$L(x)\succeq0\qquad(x\in\R^n),$$
or simply $L$, \emph{infeasible} if $S_L=\emptyset$. In this case, call it \emph{strongly infeasible} if
$$\dist\big(\{L(x)\mid x\in\R^n\},\ \smpsd\big)>0,$$ and \emph{weakly infeasible} otherwise.
Moreover, call $L$ \emph{feasible} if $S_L\neq\emptyset$. 
We say $L$ is \emph{strongly feasible} if
there is an $x\in\R^n$ such that $L(x)\succ0$ and \emph{weakly feasible} otherwise.
\end{definition}

Note that for any linear pencil $L$, each element of $M_L$ is a polynomial nonnegative on the spectrahedron $S_L$.
In general $M_L$ does not contain all polynomials nonnegative on $S_L$, e.g. when $L$ is diagonal and
$\dim S_L\ge 3$ \cite{sc}. For diagonal $L$, the quadratic module $M_L$ (and actually the convex cone $C_L$) contains however all \emph{linear} polynomials
nonnegative on the polyhedron $S_L$ by Farkas' lemma. For non-diagonal linear pencils
$L$ even this can fail, see Example \ref{new2} below.
To certify nonnegativity of a linear polynomial on a spectrahedron, we therefore employ more involved algebraic certificates
(motivated by the real radical), see Theorem \ref{sosram}.

\section{Spectrahedra and quadratic modules}\label{sec:qm}

In this section we establish algebraic certificates
for infeasibility and boundedness of a spectrahedron. They involve
classical Positivstellensatz-like certificates expressing $-1$ as a 
weighted sum of squares.

A quadratic module $M\subseteq\rx$ 
is said to be \textit{archimedean} if 
$$\forall f\in\rx \ \exists N\in\N:\; N-f^2 \in M.$$
Equivalently, there is an $N\in\N$ with $N\pm X_i\in M$ for $i=1,\ldots, n$;
see \cite[Corollary 5.2.4]{ma}.

\subsection{Quadratic modules describing non-empty bounded spectrahedra}

Obviously, if $M_L$ is archimedean for a linear pencil $L$, then
$S_L$ is bounded. In \cite{hkm} complete positivity was used
to deduce that for \emph{strictly feasible} linear pencils the converse
holds. Subsequently, a certain generalization of this result for \emph{projections} of
spectrahedra has been proved by other techniques in \cite{gn}.
In this section we will establish the result for arbitrary bounded $S_L$. We begin with the relatively
easy case of non-empty $S_L$ (possibly with empty interior).

\begin{lemma}\label{lem:compArch}
Let $L$ be a linear pencil with $S_L\neq\emptyset$. Then
$$
S_L\text{ is bounded}\quad\iff\quad M_L\text{ is archimedean.}
$$
\end{lemma}

\begin{proof}
The direction $(\Leftarrow)$ is obvious as remarked above.
Let us consider the converse.
We first establish the existence of finitely many linear polynomials
in $C_L$ certifying the boundedness of $S_L$.

There is a ball $B\subseteq\R^n$ with $S_L\subseteq B$ and $S_L\cap\partial B=\emptyset$. For every $x\in\partial B$ there is a vector $u\in\R^n$ with 
\beq\label{eq:vecBord}
u^*L(x)u<0.
\eeq
 By continuity, \eqref{eq:vecBord} holds for all $x$ in a neighborhood
$U_x$ of $x$. From $\{U_x\mid x\in\partial B\}$ we extract by compactness a finite
subcovering 
$\{U_{x_1},\ldots, U_{x_r}\}$ of $\partial B$. Let $\ell_i:= u_i^*Lu_i\in\rx_1$ and
$$
S:=\{x\in\R^n\mid \ell_1(x)\geq0,\ldots,\ell_r(x)\geq0\}.
$$
Clearly, $S_L\subseteq S$ and $S\cap \partial B=\emptyset$. Since $S_L$ is non-empty by hypothesis
and contained in $B$, it follows that $S$ contains a point of
$B$. But then it follows from the convexity of $S$ together with $S\cap \partial B=\emptyset$ that $S\subseteq B$. In particular, $S$ is bounded.

Now every $\ell_i\in C_L\subseteq M_L$. Hence the quadratic module $M$
generated by the $\ell_i$ is contained in $M_L$. 
Choose $N\in\N$ with $N\pm X_i>0$ on $S$ for all $i$. 
Fix an $i$ and $\de\in\{-1,1\}$.
The system of linear
inequalities
$$
-N + \de x_i \geq 0, \ell_1(x)\geq0,\ldots,\ell_r(x)\geq0
$$
is infeasible. Hence by Farkas' lemma \cite{fa}, 
there are $\al_j\in\R_{\geq0}$
satisfying
\beq\label{eq:far}
-1=\al_0(-N+\de X_i)+\al_1\ell_1+\cdots+\al_r\ell_r.
\eeq
Note $\al_0\neq0$ since $S\neq\emptyset$. Rearranging terms
in \eqref{eq:far} yields $N-\de X_i\in C_L$. Since $i$ and $\de$ were arbitrary and
$C_L\subseteq M_L$, we conclude that $M_L$ is archimedean.
\end{proof}

Surprisingly, establishing Lemma \ref{lem:compArch} for
\emph{empty} spectrahedra is more involved and will occupy us in
the next subsection.

\subsection{Quadratic modules describing the empty spectrahedron}

The following is an extension of Farkas' lemma from LP to SDP due to Sturm \cite[Lemma 2.18]{st}.
We include its simple proof based on a Hahn-Banach separation argument.

\begin{lemma}[Sturm]\label{sturm}
A linear pencil $L$ is strongly infeasible if and only if $-1\in C_L$.
\end{lemma}

\begin{proof}
Suppose
$$L=A_0+\sum_{i=1}^n X_iA_i$$
is strongly infeasible, $A_i\in\sm$. Then the non-empty convex sets $\{L(x)\mid x\in\R^n\}$
and $\smpsd$ can be \emph{strictly}
separated by an affine hyperplane (since their Minkowski sums with a small ball are still disjoint and can therefore be separated
\cite[Chapter III, Theorem 1.2]{ba1}).
This means that there is a non-zero linear form $$\ell\colon\sm\to\R$$
and $\alpha\in\R$
with $\ell(\smpsd)\subseteq\R_{>\alpha}$ and
$\ell(\{L(x)\mid x\in\R^n\})\subseteq\R_{<\alpha}$.
Choose
$B\in\sm$ such that $$\ell(A)=\tr(AB)$$ for all $A\in\sm$.
Since $\ell(\smpsd)$ is bounded from below, by the self-duality of the
convex cone of positive semidefinite matrices, $0\neq B\succeq0$.
Similarly, we obtain $\ell(A_i)=0$
for $i\in\{1,\dots,n\}$. Note that $\al<0$ since $0=\ell(0)\in\R_{>\alpha}$ so we can assume
$\ell(A_0)=-1$ by scaling.
Writing $B=\sum_i u_iu_i^*$ with $u_i\in\R^m$, we obtain
$$
-1= \ell(A_0)=\ell(L(x))= \tr (L(x)\sum_i u_iu_i^*)=\sum_i u_i^* L(x)u_i.
$$
for all $x\in\R^n$.
Hence $-1=\sum_i u_i^*Lu_i\in C_L$.

Conversely, if $-1\in C_L$, i.e., $-1=c+\sum_i u_i^*Lu_i$ for some $c\ge0$ and $u_i\in\R^m$, then with $B:=\sum_iu_iu_i^*\in\smpsd$ we obtain a linear form
$$\ell\colon\sm\to\R, \quad A\mapsto \tr(AB)$$
satisfying $\ell(\smpsd)\subseteq\R_{\geq0}$ and $\ell(\{L(x)\mid x\in\R^n\})=\{-1-c\}\subseteq\R_{\le-1}$. So $L$ is strongly infeasible.
\end{proof}

\begin{lemma}\label{wind}
Let $L$ be an infeasible linear pencil of size $m$. The following are equivalent:
\begin{enumerate}[{\rm(i)}]
\item $L$ is weakly infeasible;\label{wind1}
\item $S_{L+\ep I_m}\neq\emptyset$ for all $\ep>0$.\label{wind2}
\end{enumerate}
\end{lemma}

\begin{proof}
Since all norms on a finite-dimensional vector space are equivalent, we can use the operator norm on $\m$.

Suppose that \eqref{wind1} holds and $\ep>0$ is given. Choose $A\in\smpsd$ and $x\in\R^n$ with $\|A-L(x)\|<\ep$.
Then $L(x)+\ep I_m\succeq0$, i.e., $x\in S_{L+\ep I_m}$.

Conversely, suppose that \eqref{wind2} holds. To show that $$\dist\big(\{L(x)\mid x\in\R^n\},\ \smpsd\}\big)=0,$$ we let $\ep>0$ be given and have to find
$A\in\smpsd$ and $x\in\R^n$ with $$\|L(x)-A\|\le\ep.$$ 
But this is easy: choose $x\in\R^n$ with $L(x)+\ep I_m\succeq0$, and set $A:=L(x)+\ep I_m$.
\end{proof}

The following lemma is due to Bohnenblust \cite{bo} (see also \cite[Theorem 4.2]{ba2} for an easier accessible reference). While Bohnenblust gave a non-trivial bound on the number of terms that are really
needed to test condition \eqref{bohne1} below, we will not need this improvement
and therefore take the trivial bound $m$. Then the proof becomes easy and we include
it for the convenience of the reader.

\begin{lemma}[Bohnenblust]\label{bohne}
For $A_1,\dots,A_n\in S\R^{m\times m}$ the following are equivalent:
\begin{enumerate}[\rm(i)]
\item Whenever $u_1,\dots,u_m\in\R^m$ with $\sum_{i=1}^mu_i^*A_ju_i=0$ for all
$j\in\{1,\dots,n\}$, then $u_1=\dots=u_m=0$;\label{bohne1}
\item $\Span(A_1,\dots,A_n)$ contains a positive definite matrix.\label{bohne2}
\end{enumerate}
\end{lemma}

\begin{proof}
It is trivial that \eqref{bohne2} implies \eqref{bohne1}. To prove that \eqref{bohne1} implies \eqref{bohne2}, note that
$\smpsd=\{\sum_{i=1}^mu_iu_i^*\mid u_1,\dots,u_m\in\R^m\}$ and $\sum_{i=1}^mu_i^*Au_i=\tr(A\sum_{i=1}^mu_iu_i^*)$ for all
$u_1,\dots,u_m\in\R^m$. The hypotheses thus say exactly that, given any $B\in\smpsd$, we have 
\beq\label{eq:bbl1}
\tr(A_1B)=\dots=\tr(A_nB)=0\implies B=0.
\eeq
Now suppose that $\Span(A_1,\dots,A_n)\cap\smpd=\emptyset$. By the standard separation theorem for two non-empty disjoint convex sets
(see for example \cite[Chapter III, Theorem 1.2]{ba1}), $\Span(A_1,\dots,A_n)$ and $\smpd$ can be separated by a hyperplane (the separating affine
hyperplane must obviously contain the origin). Therefore there is a non-zero linear form $L\colon\sm\to\R$ with 
$$L(\smpd)\subseteq\R\nn\quad\text{and}\quad
L(\Span(A_1,\dots,A_n))\subseteq\R_{\le0}.$$ 
Then of course $L(\smpsd)\subseteq\R\nn$ and $L(\Span(A_1,\dots,A_n))=\{0\}$. Now choose
$B\in\sm$ such that 
$$L(A)=\tr(AB)\quad\text{for all}\quad A\in\sm.$$ 
Then $B\neq 0$, $B\in\smpsd$ and $\tr(A_1B)=\dots=\tr(A_nB)=0$,
contradicting \eqref{eq:bbl1}.
\end{proof}

\begin{lemma}\label{kaffeebohne}
Let $L$ be a linear pencil of size $m$ which is either weakly infeasible or weakly feasible.
Then there are $k\ge1$ and $u_1,\dots,u_k\in\R^m\setminus\{0\}$
such that $\sum_{i=1}^ku_i^*Lu_i=0$.
\end{lemma}

\begin{proof}
Assume that the conclusion is not true.
By Lemma \ref{bohne} there are hence $x_0,x_1,\dots,x_n\in\R$ such that $$x_0A_0+x_1A_1+\dots+x_nA_n\succ0.$$
Of course it is impossible that $x_0>0$ since otherwise $L(0)\succ0$.
Also $x_0=0$ is excluded (since otherwise
$L(cx_1,\dots,cx_n)\succ0$ for $c>0$ large enough). Hence without loss of generality $x_0=-1$, i.e.,
$x_1A_1+\dots+x_nA_n\succ A_0$. Choose $\ep>0$ such that
$$x_1A_1+\dots+x_nA_n\succ A_0+2\ep I_m.$$ By Lemma \ref{wind}, we can choose some $y\in S_{L+\ep I_m}$. But then
$$A_0+(x_1+2y_1)A_1+\dots+(x_n+2y_n)A_n\succ2(A_0+\ep I_m+y_1A_1+\dots+y_nA_n)\succeq0,$$ contradicting the
hypotheses.
\end{proof}

We come now to our first main theorem which is a version of Farkas' lemma for SDP which unlike Lemma \ref{sturm} does not only work for
strongly but also for weakly infeasible linear pencils. The price we pay is that we 
have to replace the finite-dimensional convex cone $C_L$ by the
infinite-dimensional quadratic module $M_L$. 
We will spend most of the sequel to refine and tame this result,
cf.~Theorems \ref{thm:bounds1} and
\ref{sosram} below.

\begin{theorem}\label{thm:main1}
Let $L$ be a linear pencil. Then $$S_L=\emptyset\quad\iff\quad -1\in M_L.$$
\end{theorem}

\begin{proof}
One direction is trivial since $M_L$ contains only polynomials nonnegative on $S_L$. We show the nontrivial implication from left to right
by induction on the number $n$ of variables appearing in $L$. In the case $n=0$, we have $L\in\sm$ and therefore
$S_L=\emptyset\iff L\not\succeq0\iff-1\in M_L$.

For the induction step suppose $n\ge 1$ and the statement is already known for linear pencils in less than $n$ variables. The way we will
later use the induction hypothesis is explained by the following claims which are weakening of the theorem itself.

Note that Claim 2 is just a special case of Claim 3' which we formulate for the
sake of clarity. The proof of Claim 3' is very conceptual and uses a tool of real algebraic geometry, namely Prestel's semiorderings \cite{pd}.
For readers that are not familiar with this theory we prove a special case of it by an ad hoc technical argument. This special case
of Claim 3' is stated in Claim 3 and suffices for our subsequent application.

\smallskip
{\bf Claim 1.} If $S_L=\emptyset$ and $0\neq\ell\in\rxlin$, then $-1\in M_L+(\ell)$.

\emph{Explanation.} After an affine linear variable transformation, we may assume $\ell=X_n$. Set $L':=L(X_1,\dots,X_{n-1},0)$. Then
$-1\in M_L+(X_n)$ is equivalent to $-1\in M_{L'}$. By induction hypothesis, the latter follows if $S_{L'}=\emptyset$. But emptiness of $S_L$ clearly
implies the emptiness of $S_{L'}$.\eop

\smallskip
{\bf Claim 2.} If $S_L=\emptyset$ and $0\neq\ell\in M_L\cap-M_L$, then $-1\in M_L$.

\emph{Explanation.} This follows from Claim 1 using the fact that $M_L\cap-M_L$ is an ideal in $\rx$ which follows from \eqref{id4}.\eop

\smallskip
{\bf Claim 3.} If $S_L=\emptyset$ and $0\neq\ell\in\rxlin$ such that $-\ell^2\in M_L$, then $-1\in M_L$.

\emph{Explanation.}  By Claim 1, there is $p\in M_L$ and $q\in\rx$ such that $-1=p+q\ell$. Now
\bes
\begin{split}
-1 & =-2+1=(2p+2q\ell)+((1+q\ell)^2-2q\ell+q^2(-\ell^2))\\
&= 2p+q^2(-\ell^2)+(1+q\ell)^2\in M_L,
\end{split}
\ees
as desired.
\eop

\smallskip
{\bf Claim 3'.} If $S_L=\emptyset$, $k\ge 1$ and $\ell_1,\dots,\ell_k\in\rxlin$ such that $0\neq\ell_1\dotsm\ell_k\in M_L\cap-M_L$, then $-1\in M_L$.

\emph{Explanation.} Assume $-1\notin M_L$. Then there is a semiordering $S$ of $\rx$ such that $M_L\subseteq S$. The advantage of $S$ over
$M_L$ is that $S\cap-S$ is always a prime ideal. Hence $\ell_i\in S\cap-S$ for some $i\in\{1,\dots,k\}$. But this shows that $-1\notin M_L+(\ell_i)$
(even though $(\ell_i)$ might not be contained in $M_L\cap-M_L$). Hence $S_L\neq\emptyset$ by Claim~1.\eop

\smallskip
Now we are prepared to do the induction step. Suppose $S_L=\emptyset$. If $L$ is strongly infeasible, then $-1\in C_L\subseteq M_L$ by Lemma
\ref{sturm}. Therefore we assume from now on that $L$ is weakly infeasible. By Lemma \ref{kaffeebohne}, one of the
following two cases applies:

\smallskip
{\bf Case 1.} There is $u\in\R^m\setminus\{0\}$ with $u^*Lu=0$.

Of course, we can assume that $L$ has no zero column since otherwise we can remove it together with the corresponding row
without altering $S_L$ and $M_L$. Write $L=(\ell_{ij})_{1\le i,j\le m}$. By changing coordinates on $\R^m$, we can assume that $u$ is the
first standard basis vector, i.e., $\ell_{11}=0$. Moreover, we may assume $\ell_{12}\neq0$. Then
$$2pq\ell_{12}+q^2\ell_{22}\in M_L$$ for all $p,q\in\rx$. In particular, $\ell_{12}^2(\ell_{22}+2p)\in M_L$ for all $p\in\rx$ and therefore
$-\ell_{12}^2\in M_L$. Now we conclude by Claim 3 or 3' that $-1\in M_L$.

\smallskip
{\bf Case 2.} Case 1 does not apply but there are $k\ge 2$ and $u_1,\dots,u_k\in\R^m\setminus\{0\}$ such that $\sum_{i=1}^ku_i^*Lu_i=0$.

In this case, $$\ell:=u_1^*Lu_1=-\sum_{i=2}^ku_i^*Lu_i\in C_L\cap-C_L\subseteq M_L\cap-M_L$$ and we are done by Claim 2, since
$\ell=u_1^*Lu_1\neq 0$.
\end{proof}

\begin{corollary}\label{cor:compArch}
Let $L$ be a linear pencil. 
Then
$$
S_L\text{ is bounded}\quad\iff\quad M_L\text{ is archimedean.}
$$
\end{corollary}
              
\begin{proof}
If $S_L\neq\emptyset$, then this is Lemma \ref{lem:compArch}.
If $S_L=\emptyset$, then $-1\in M_L$ by Theorem \ref{thm:main1},
so $M_L$ is archimedean.
\end{proof}     

\begin{rem}
Note that the above corollary is a strong variant
of Schm\"udgen's characterization \cite{sm} of bounded basic closed
semialgebraic sets as being exactly those whose describing
finitely generated preorderings are archimedean. Preorderings have
the tendency of being much larger than quadratic modules.
In general, a finitely generated quadratic module might describe a
bounded or even an empty set without being archimedean, see
\cite[Example 6.3.1]{pd} and \cite[Example 7.3.2]{ma}.
Corollary~\ref{cor:compArch} says that quadratic modules associated to linear pencils behave very well
in this respect.
\end{rem}

We conclude this section with a version of Putinar's Positivstellensatz \cite{pu} for bounded spectrahedra:

\begin{cor}\label{linpos}
Let $L$ be a linear pencil and assume that $S_L$ is bounded.
If $f\in\R\cx$ satisfies $f|_{S_L}>0$, then $f\in M_L$.
\end{cor}

\begin{proof}
By Corollary \ref{cor:compArch}, $M_L$ is archimedean. Now apply (a slight generalization of) Putinar's Positivstellensatz
\cite[Theorem 5.6.1]{ma}.
\end{proof}

\begin{remark} Let $L$ be a linear pencil with bounded $S_L$.
\begin{enumerate}[(1)]
\item In the case $L$ is strongly feasible, Corollary \ref{linpos} has already been proved in \cite[\S 7]{hkm} by completely different techniques,
namely complete positivity from operator algebras. Note however that the more involved case in our approach occurs when $L$ is infeasible.
\item 
From Corollary $8$ it is easy to see that the quadratic module \emph{in the sense of rings with involution} (see \cite{ks}) associated to a
linear pencil
$L$ of size $m$ in the ring of $s\times s$ matrix polynomials is archimedean (in the sense of \cite[Subsection 3.1]{ks} or \cite[\S 6, \S 7]{hkm}) if the
spectrahedron $S_L$ defined by $L$ is bounded (cf.~\cite[\S 7]{hkm}). 
Among other consequences, this implies a suitable generalization of
Corollary \ref{linpos} for \emph{matrix polynomials} positive definite on the bounded spectrahedron $S_L$ (cf.~\cite[Corollary 1]{hs},
\cite[Theorem 3.7]{ks} and \cite[Theorem 7.1]{hkm}).
\end{enumerate}
\end{remark}

\subsection{Degree Bounds}\label{subs:degbound}

Given a linear pencil $L$ of size $m$ and $k\in\N_0$, let
\begin{align*}
M_L^{(k)}&:=\Big\{\sum_ip_i^2+\sum_jv_j^*Lv_j\mid p_i\in\rx_k,~v_j\in\rx^m_k\Big\}\\
&~=\Big\{s+\tr(LS)\mid s\in\R\cx_{2k}\text{ sos-polynomial},\ S\in\R\cx^{m\times m}_{2k}\text{ sos-matrix}\Big\}\\
&~\subseteq\R\ac_{2k+1},
\end{align*}
be the \emph{truncated quadratic module with degree restriction} $k$
associated to $L$. Note that $M_L^{(0)}=C_L$.

By a closer inspection of the proof of Theorem \ref{thm:main1} we
can obtain exponential degree bounds for the infeasibility certificate.

\begin{theorem}\label{thm:bounds1}
Let $L$ be an infeasible linear pencil in $n$ variables. Then
$$-1\in M_L^{(2^n-1)}.$$
\end{theorem}

\begin{proof}
We shall prove this by induction on $n$. The statement is
clear for $n=0$. Given $n\in\N$ and $$L=A_0+\sum_{i=1}^n X_iA_i,$$ we assume the statement has been
established for all infeasible linear pencils with  $n-1$ variables.

By Lemma \ref{kaffeebohne}, there is an $u\in\R^{m^2}\setminus\{0\}$ with $u^*(I_m\otimes L)u=0$.
Replacing $L$ by $I_m\otimes L$  changes neither $M_L^{(k)}$ nor $S_L=\emptyset$. Without loss of generality,
we assume therefore that there is $u\in\R^m\setminus\{0\}$ with $u^*Lu=0$. Writing $L=(\ell_{ij})_{1\le i,j\le m}$
and performing a linear coordinate change on $\R^m$, we can moreover assume $\ell_{11}=0$.
Furthermore, without loss of generality, $\ell_{12}\neq0$. As in
Case~1 of the proof of Theorem \ref{thm:main1} above, we deduce 
\beq\label{eq:ugly0}
-\ell_{12}^2\in M_L^{(1)}.
\eeq
If $\ell_{12}\in\R$, we are done. Otherwise after possibly performing an
affine linear change of variables on $\R^n$, we may assume $\ell_{12}=X_n$.

Now $L':=L(X_1,\ldots,X_{n-1},0)$ is an 
infeasible linear pencil in $n-1$ variables.
By our induction hypothesis, $-1\in M_{L'}^{(2^{n- 1}-1)}$.
In particular, there are $$p_i\in\rx_{2^{n-{ 1}}-1}\text{ and }v_j\in\rx^m_{2^{n-{ 1}}-1}$$ satisfying
$$
-1=\sum_ip_i^2+\sum_jv_j^*L'v_j.
$$
Let $q:=\sum_j v_j^*A_nv_j\in\rx_{2^{n}-2}$.
Then
\beq\label{eq:ugly1}\begin{split}
-1&=2\sum_ip_i^2+2\sum_jv_j^*L'v_j+1\\
&=2\sum_ip_i^2+2\sum_jv_j^*Lv_j-2qX_n+1\\
& = 2\sum_ip_i^2+2 \sum_jv_j^*Lv_j + (1-qX_n)^2+q^2(-X_n^2).
\end{split}\eeq
Since $\deg q\leq 2^{n}-2$, we have $q^2(-X_n)\in M_L^{(2^{n}-1)}$
by \eqref{eq:ugly0}.
Taken together with $(1-qX_n)^2\in  M_L^{(2^{n}-1)}$,
\eqref{eq:ugly1} implies $-1\in M_L^{(2^{n}-1)}$.
\end{proof}

Similarly as in Theorem \ref{thm:bounds1}, one can obtain 
a bound on $k$ in a certificate of the form $-1\in M_L^{(k)}$  
that depends exponentially on the size $m$ of $L$.

\begin{thm}\label{thm:bounds2}
Let $L$ be an infeasible linear pencil of size $m$. Then
$$-1\in M_L^{(2^{m-1}-1)}.$$
\end{thm}

\begin{proof}
We prove this by induction on $m$. The statement is
clear for $m=1$. 
Given $$L=A_0+\sum_{i=1}^n X_iA_i\in S\R\cx^{m\times m}$$ 
of size $m\ge2$,
we assume the statement has been
established for all infeasible linear pencils
 of size $m-1$.
If $L$ is strongly infeasible, then $-1\in C_L=M_L^{(0)}$
by Lemma \ref{sturm}. So we may assume $L$ is weakly infeasible.

\smallskip
{\bf Claim.} 
There is an affine linear change of variables after which
$L$ assumes the form
\bes%\label{eq:pencilUltimo}
L = \begin{bmatrix} b_0 & b^* \\ b & L'
\end{bmatrix},
\ees
where
$b_0\in\R\cx_1$,
$b=\begin{bmatrix} b_1 & \cdots & b_{m-1}\end{bmatrix}^*\in\R\cx_1^{m-1}$, 
$L'$ is a linear pencil of size $m-1$, and $b_j\in\R\cx_1$ satisfy
\beq\label{eq:bj}
-b_j^2\in M_L^{(1)} \quad \text{for }j=0,\ldots,m-1.
\eeq
Furthermore, $b_0$ can be chosen to be either $0$ or $X_1$.

\emph{Explanation.}
By Lemma \ref{kaffeebohne},
there is $k\in\N$ and $u_1,\ldots,u_k\in \R^m\setminus\{0\}$ with
$\sum_{i=1}^k u_i^*Lu_i=0$. We distinguish two cases.

\smallskip
{\bf Case 1.} There is $u\in\R^m\setminus\{0\}$ with $u^*Lu=0$.

%Of course, we can assume that $L$ has no zero column (or row) since otherwise we can delete it without altering $S_L$ and $M_L$. 
Write $L=(\ell_{ij})_{1\le i,j\le m}$. By changing coordinates on $\R^m$, we can assume that $u$ is the
first standard basis vector, i.e., $\ell_{11}=0$. 
Hence
$$
L=\begin{bmatrix}
0& b^* \\
b & L'
\end{bmatrix},
$$
where $b=\begin{bmatrix} b_1 & \cdots & b_{m-1}\end{bmatrix}^*\in\R\cx_1^{m-1}$ and 
$L'$ is a linear pencil of size $m-1$.
As in the proof of Theorem \ref{thm:bounds1},
we deduce
$-b_j^2\in M_L^{(1)}$ for all $j=1,\ldots,m-1$.

\smallskip
{\bf Case 2.}
Case 1 does not apply but there are $k\ge 2$ and $u_1,\dots,u_k\in\R^m\setminus\{0\}$ such that $\sum_{i=1}^ku_i^*Lu_i=0$.

In this case, $$\ell_{11}:=u_1^*Lu_1=-\sum_{i=2}^ku_i^*Lu_i\in C_L\cap-C_L=
M_L^{(0)}\cap-M_L^{(0)}.$$
Since Case 1 does not apply, $\ell_{11}\neq 0$. Furthermore, since $L$
is assumed to be weakly infeasible, $\ell_{11}\not\in\R$.
Hence
after an affine linear change of variables on $\R^n$, we
can assume $\ell_{11}=X_1$.
Thus
\bes%q\label{eq:afterChg1}
L=\begin{bmatrix}
X_1& b^* \\
b & L'
\end{bmatrix},
\ees
where $b=\begin{bmatrix} b_1 & \cdots & b_{m-1}\end{bmatrix}^*\in\R\cx_1^{m-1}$ and 
$L'$ is a linear pencil of size $m-1$.
Note that 
\[
-4 X_1^2 =  (1-X_1)^2 X_1 +  (1+X_1)^2 (-X_1)
\]
shows that $-X_1^2\in M_L^{(1)}$.
Using this, one gets similarly as above that also each of the entries $b_j$ of
$b$ satisfies $-b_j^2\in M_L^{(1)}$.
This proves our claim. \eop

\smallskip
If one of the $b_j\in\R\setminus\{0\}$, we are done
by \eqref{eq:bj}.
Otherwise we consider two cases.

\smallskip
{\bf Case a.}
If the linear system $b_0(x)=0,\ b(x)=0$ is infeasible, then 
we proceed as follows.
There are $\alpha_0,\ldots,\al_{m-1}\in\R$ satisfying
\beq\label{eq:fart}
\sum_{j=0}^{m-1} \al_j b_j = 1.
\eeq
For each $j=0,\ldots,m-1$ and $\de\in\R$ we have
\bes%q\label{eq:de}
1+\de b_j = \Big( 1+\frac {\de}2 b_j\Big)^2 +
\frac{\de^2}4 (-b_j^2) \in M_L^{(1)}
\ees
by \eqref{eq:bj}.
Hence \eqref{eq:fart} implies
\[
-1 = 1-2 = 1- 2\sum_{j=0}^{m-1} \al_j b_j = 
\sum_{j=0}^{m-1} \Big( \frac 1{m} - 2\al_jb_j\Big) \in M_L^{(1)}.
\]

\smallskip
{\bf Case b.}
Suppose the linear system $b_0(x)=b(x)=0$ is feasible. 
Then we perform an affine linear change of variables on $\R^n$ to ensure
\[
\{x\in\R^n\mid b_0(x)=0,\ b(x)=0\} = \{0\}^r \times \R^{n-r}
\]
for some $r\in\N$.
Moreover, we may assume $X_{1},\ldots,X_{r}$ are among the
entries $b_j$, $j=0,\ldots,m-1$.

Now $L'':=L'(0,\ldots,0,X_{r+1},\ldots, X_n)$ is an infeasible
linear pencil of size $m-1$. By our induction hypothesis,
$-1\in M_{L''}^{(2^{m-2}-1)}$.
In particular, there are $s\in\sos$ with $\deg s\leq 2^{m-1}-2$, and
$v_i\in\rx^{m-1}_{2^{m-2}-1}$ satisfying
$$
-1=s+\sum_i v_i^*L'' v_i.
$$
Introducing
\[
q_t:=\sum_i v_i^*A_tv_i \in\rx_{2^{m-1}-2}\qquad\text{and}\qquad w_i:=\begin{bmatrix}0\\v_i\end{bmatrix}\in\rx^m_{2^{m-2}-1}
\]
we have
\beq\label{eq:ugly3c}
\begin{split}
-1 &=\big(2s+2\sum_iv_i^*L''v_i\big)+1\\
&=\big( 2s + 2 \sum_iv_i^*L'v_i-
\sum_{t=1}^r 2q_tX_t \big)+
\sum_{t=1}^r 
\Big( \big( \frac1{\sqrt r}-\sqrt rq_tX_t\big)^2+2q_tX_t+r q_t^2(-X_t^2)\Big) \\
& = 2s + 2 \sum_{i}  w_i^*L w_i + \sum_{t=1}^r\big(\frac1{\sqrt r}-\sqrt rq_tX_t\big)^2+\sum_{t=1}^rrq_t^2(-X_t^2).
\end{split}\eeq
Combining 
$q_t^2(-X_t^2)\in M_L^{(2^{m-1}-1)}$
with $(\frac1{\sqrt r}-\sqrt rq_tX_t)^2\in  M_L^{(2^{m-1}-1)}$,
\eqref{eq:ugly3c} implies $-1\in M_L^{(2^{m-1}-1)}$.
\end{proof}

\begin{cor}\label{cor:bounds}
Let $L$ be an infeasible linear pencil of size $m$ in $n$ variables. Then
$$-1\in M_L^{(2^{\min\{m{ -1},n\}}-1)}.$$
\end{cor}

\begin{proof}
This is immediate from Theorems \ref{thm:bounds1} and \ref{thm:bounds2}.
\end{proof}

\subsection{Examples}\label{subs:ex}

The standard textbook example \cite{st,wsv} of a weakly infeasible linear pencil
seems to be
$$
L:= \begin{bmatrix} X & 1 \\ 1 & 0\end{bmatrix}.
$$
Then $-1\not\in C_L$, but $-1\in M_L^{(1)}$.
Indeed,
for $$u:=\begin{bmatrix}1 & -1-\frac X2\end{bmatrix}^*,$$ we have
$$-1=\frac12u^*Lu.$$

\begin{example}\label{ex:sturm2}
Let 
$$
L:= \begin{bmatrix}0& X_1 & 0 \\ X_1 & X_2&1\\ 0&1&X_1\end{bmatrix}.
$$
Then $L$ is weakly infeasible and $-1\not\in M_L^{(1)}$.

Assume otherwise, and let 
\beq\label{eq:contr1}
-1=s+\sum_jv_j^*Lv_j,
\eeq where
$v_j\in\rx_1^3$ and $s\in\sum\rx^2$ with $\deg s\leq2$. 
We shall carefully analyze the terms $v_j^*Lv_j$. Write
$$v_j=\begin{bmatrix}q_{1j}&q_{2j}&q_{3j}\end{bmatrix}^*,
\quad \text{and}\quad q_{ij}=a_{ij}+b_{ij}X_1+c_{ij}X_2$$
with $a_{ij},b_{ij},c_{ij}\in\R$.
Then the $X_2^3$ coefficient of $v_j^*Lv_j$ equals
$c_{2j}^2$, so $c_{2j}=0$ for all $j$.
Next, by considering the $X_1X_2^2$ terms, we deduce $c_{3j}=0$.
Now the only terms possibly contributing to $X_2^2$ come from $s$, so
$s\in\R[X_1]_2$. The coefficient of $X_2$ in $v_j^*Lv_j$ is a square,
so $a_{2j}=0$. But now $v_j^*Lv_j$ does not have a constant term
anymore, leading to a contradiction with \eqref{eq:contr1}. 

From Theorem \ref{thm:main1} it follows that $-1\in M_L^{(3)}$. 
In fact, $-1\in M_L^{(2)}$ since
\[
-2 = u^* L u
\]
for 
$u= \begin{bmatrix}
\frac{1}{2}+ \frac{X_2}{2}+\frac{X_2^2}{8} & -1 & 1+\frac
   {X_2}{2}
\end{bmatrix}^*\in \R\cx_2^3$.
\end{example}

\section{Polynomial size certificates for infeasibility}\label{sec:ram}

Theorem \ref{thm:bounds1} (or Corollary \ref{cor:bounds}) enables us to reformulate the feasibility of a linear pencil as the infeasibility of an LMI.
Indeed, let a linear pencil $L$ in $n$ variables of size $m$ be given. Then $L$ is infeasible if and only if there exists an sos-polynomial
$s\in\rx$ and an sos-matrix $S\in\rx^{m\times m}$ both of degree at most $2^{n+1}-2$ such that
\begin{equation}\label{eq:e-1}
-1=s+\tr(LS).
\end{equation}
By comparing coefficients, the polynomial equation \eqref{eq:e-1} can of course be written as a system of linear equations in the coefficients of $s$
and $S$. Now (the coefficient tuples of) sos-polynomials in $\rx$ of bounded degree form \emph{a projection of} a spectrahedron.
In other words, the condition of being (the coefficient tuple) of an sos-polynomial in $\rx$ of bounded degree can be expressed within an LMI
by means of additional variables. This is the well known \emph{Gram matrix method} \cite{lau,ma}. As noted by Kojima \cite{ko} and nicely
described by Hol and Scherer \cite{hs}, the Gram matrix method extends easily to sos-\emph{matrices}.

Hence  the existence of $s$ and $S$ in \eqref{eq:e-1} is indeed equivalent to the feasibility of an LMI which we will not write down explicitly.
The drawback of this LMI is that it will be large since our degree bounds for the sos-polynomial and the sos-matrix are exponential.

This problem is overcome in this section, where for a given linear pencil $L$, we construct an LMI whose feasibility is equivalent to the infeasibility
of the given linear pencil $L$ and which can be written down in polynomial time (and hence has polynomial size) in the bit size of $L$ if $L$
has rational coefficients. Each feasible point of the new LMI gives rise to a certificate of infeasibility for $L$ which is however more involved than the one
using the quadratic module $M_L$. In addition to the quadratic module, we will now use another notion from real algebraic geometry, namely the
real radical ideal.

But actually we will not only characterize infeasibility of LMIs but even give a new duality theory for SDP where strong duality always holds in contrast
to the standard duality theory. We will call our dual the \emph{sums of squares dual} of an SDP. For a given (primal) SDP with rational coefficients the sums of squares dual can be written down in polynomial time in the bit size of the primal.

\subsection{Review of standard SDP duality}\label{subs:standard}

We first recall briefly the standard duality theory of SDP. We present it from the view point of a real algebraic geometer, i.e., we use the language
of polynomials in the formulation of the primal-dual pair of SDPs and in the proof of strong duality. This is necessary for a good understanding of
the sums of squares dual which we will give later.

A semidefinite program (P) and its standard dual (D) is given by a linear pencil $L\in\R[\x]^{m\times m}$ and
a linear polynomial $\ell\in\R[\x]$ as follows:

\bigskip
$
\begin{array}[t]{lrl}
(P)&\text{minimize}&\ell(x)\\
&\text{subject to}&x\in\R^n\\
&&L(x)\succeq0
\end{array}
$\hfill
$
\begin{array}[t]{lrl}
(D)&\text{maximize}&a\\
&\text{subject to}&S\in\sm,\ a\in\R\\
&&S\succeq0\\
&&\ell-a=\tr(LS)
\end{array}
$

\bigskip\noindent
To see that this corresponds (up to some minor technicalities) to the formulation in the literature, just write the polynomial constraint
$\ell-a=\tr(LS)$ of the dual as $n+1$ linear equations by comparing coefficients.

The optimal values of $(P)$ and $(D)$ are defined to be
\begin{align*}
P^*&:=\inf\{\ell(x)\mid x\in\R^n,\ L(x)\succeq0\}\in\R\cup\{\pm\infty\}\qquad\text{and}\\
D^*&:=\sup\{a\mid S\in S\R^{m\times m}_{\succeq0},\ a\in\R,\ \ell-a=\tr(LS)\}\in\R\cup\{\pm\infty\},
\end{align*}
respectively, where the infimum and the supremum is taken in the ordered set $\{-\infty\}\cup\R\cup\{\infty\}$ (where $\inf\emptyset=\infty$ and
$\sup\emptyset=-\infty$).

By \emph{weak duality}, we mean that $P^*\ge D^*$ or equivalently that the objective value of $(P)$ at any of its feasible points is
greater or equal to the objective value of $(D)$ at any of its feasible points.

Fix a linear pencil $L$. It is easy to see that weak duality holds for all primal objectives $\ell$ if and only if
$$f\in C_L\implies\text{$f\ge0$ on $S_L$}$$
holds for all $f\in\rx_1$, which is of course true.
By \emph{strong duality}, we mean that $P^*=D^*$ (\emph{zero duality gap}) and that (the objective of) $(D)$ \emph{attains} this common optimal value
in case this is finite. It is a little exercise to see that strong duality for all primal objectives $\ell$ is equivalent to
$$\text{$f\ge0$ on $S_L$}\iff f\in C_L$$
for all $f\in\rx_1$.

As we have seen in Subsection \ref{subs:ex}, this fails in general.
It is however well-known that it is true when the feasible set $S_L$ of the primal (P) has non-empty interior
(e.g. if $L$ is strongly feasible). For convenience of the reader, we include a proof which has a bit of a flavor
of real algebraic geometry.

\begin{prop}[Standard SDP duality]\label{usual}
Let $L\in S\rx^{m\times m}$ be a linear pencil such that $S_L$ has non-empty interior. Then
$$f\ge0\text{\ on\ }S_L\ \iff\ f\in C_L$$
for all $f\in\rx_1$.
\end{prop}

\begin{proof}
In a preliminary step, we show that the convex cone $C_L$ is closed in $\rx_1$. To this end,
consider the linear subspace $U:=\{u\in\R^m\mid Lu=0\}\subseteq\R^m$.
The map
$$\ph\colon\R\times(\R^m/U)^m\to C_L,\quad(a,\bar u_1,\dots,\bar u_m)\to a^2+\sum_{i=1}^mu_i^*Lu_i$$
is well-defined and surjective.

Suppose $\ph$ maps $(a,\bar u_1,\dots,\bar u_m)\in(\R^m/U)^m$ to $0$. Fix $i\in\{1,\dots,m\}$. Then
$u_i^*L(x)u_i=0$ for all $x\in S_L$. Since $L(x)\succeq0$, this implies $L(x)u_i=0$ for all
$x\in S_L$. Using the hypothesis that $S_L$ has non-empty interior, we conclude that $Lu_i=0$, i.e., $u_i\in U$.
Since $i$ was arbitrary and $a=0$, this yields $(a,\bar u_1,\dots,\bar u_m)=0$.

This shows $\ph^{-1}(0)=\{0\}$. Together with the fact that $\ph$ is a (quadratically) homogeneous map, this implies that
$\ph$ is proper (see for example \cite[Lemma 2.7]{ps}). In particular, $C_L=\im\ph$ is closed.

Suppose now that $f\notin\rx_1\setminus C_L$. The task is to find $x\in S_L$ such that $f(x)<0$.
Being a closed convex cone, $C_L$ is the intersection of all closed half-spaces containing
it. Therefore we find a linear map $\ps:\rx_1\to\R$ such that $\ps(C_L)\subseteq\R\nn$ and $\ps(f)<0$. We can assume
$\ps(1)>0$ since otherwise $\ps(1)=0$ and we can replace $\ps$ by $\ps+\ep\ev_y$ for some small $\ep>0$ where
$y\in S_L$ is chosen arbitrarily. Hereby $\ev_x\colon\rx_1\to\R$ denotes the evaluation in $x\in\R^n$. Finally, after a
suitable scaling we can even assume $\ps(1)=1$.

Now setting $x:=(\ps(X_1),\dots,\ps(X_n))\in\R^n$, we have $\ps=\ev_x$. So $\ps(C_L)\subseteq\R\nn$ means exactly that
$L(x)\succeq0$, i.e., $x\in S_L$. At the same time $f(x)=\ps(f)<0$ as desired.
\end{proof}

\subsection{Certificates for low-dimensionality of spectrahedra}\label{subs:lowdim}

As mentioned above, the problems with the standard duality theory for SDP arise when one deals with spectrahedra having empty interior.
Every convex set with empty interior is contained in an affine hyperplane. The basic idea is now to code the search for such an affine hyperplane
into the dual SDP and to replace equality in the constraint $f-a=\tr(LS)$ of (D) by congruence modulo the linear
polynomial $\ell$ defining the affine hyperplane. This raises however several issues:

First, $S_L$ might have codimension bigger than one in $\R^n$. This will be resolved by iterating the search up to $n$ times.

Second, we do not see any possibility to encode the search for the linear polynomial $\ell$ directly into an SDP. What we can implement is the search
for a non-zero \emph{quadratic} sos-polynomial $q$ together with a certificate of $S_L\subseteq\{q=0\}$. Note that $\{q=0\}$ is a proper affine subspace
of $\R^n$. It would be ideal to find $q$ such that $\{q=0\}$ is the affine hull of $S_L$ since then we could actually avoid the $n$-fold iteration just
mentioned. However it will follow from Example \ref{ex:new} below that this is in general not possible.

Third, we have to think about how to implement congruence modulo linear polynomials $\ell$ vanishing on $\{q=0\}$. This will be dealt with by
using the radical ideal from real algebraic geometry in connection with Schur complements.

We begin with a result which ensures that a suitable quadratic sos-polynomial $q$ can always be found. In fact, the following proposition says that
there exists such a $q$ which is actually a square. The statement is of interest in itself since it provides certificates for low-dimensionality of
spectrahedra. We need \emph{quadratic} (i.e., degree $\le 2$) sos-matrices for this.

\begin{proposition}\label{exsos}
For any linear pencil $L\in S\rx^{m\times m}$, the following are equivalent:
\begin{enumerate}[\rm(i)]
\item $S_L$ has empty interior;\label{exsos1}
\item There exists a non-zero linear polynomial\label{exsos2}
$\ell\in\rx_1$ and a quadratic sos-matrix $S\in S\rx^{m\times m}$ such that
\begin{equation}\label{eq:lowdim}
-\ell^2=\tr(LS).
\end{equation}
\end{enumerate}
\end{proposition}

\begin{proof}
From \eqref{exsos2} it follows that $-\ell^2\in M_L$ and therefore $-\ell^2\ge0$ on $S_L$, which implies $\ell=0$ on $S_L$.
So it is trivial that \eqref{exsos2} implies \eqref{exsos1}.

For the converse, suppose that $S_L$ has empty interior.
If there is $u\in\R^m\setminus\{0\}$ such that $Lu=0$ then, by an orthogonal change of
coordinates on $\R^m$, we could assume that $u$ is the first unit vector $e_1$. But then we delete the first
column and the first row from $L$. We can iterate this and therefore assume from now on that there is no
$u\in\R^m\setminus\{0\}$ with $Lu=0$.

We first treat the rather trivial case where $L$ is strongly infeasible. By Lemma~\ref{sturm}, there are
$c\in\R\nn$ and $u_i\in\R^m$ with $-1-c=\sum_iu_i^*Lu_i$. By scaling the $u_i$ we can assume $c=0$.
Setting $S:=\sum_iu_iu_i^*\in\sm$ and $\ell:=1$, we have $-\ell^2=-1=\sum_iu_i^*Lu_i=\tr(LS)$ for the
constant sos-matrix $S$ and the constant non-zero linear polynomial $\ell$.

Now we assume that $L$ is weakly infeasible or feasible. In case that $L$ is feasible, it is clearly weakly feasible
since otherwise $S_L$ would have non-empty interior. Now Lemma \ref{kaffeebohne} justifies the
following case distinction:

\smallskip
{\bf Case 1.} There is $u\in\R^m\setminus\{0\}$ with $u^*Lu=0$.

Write $L=(\ell_{ij})_{1\le i,j\le m}$. Again by an orthogonal change of coordinates on $\R^m$, we can assume that $u=e_1$,
i.e., $\ell_{11}=0$. Moreover, we may assume $\ell:=\ell_{12}\neq0$ (since $Le_1=Lu\neq0$). Setting
$\ell':=\frac12(-1-\ell_{22})$,
$v:=[\ell'\ \ell\ 0\dots0]^*$ and $S:=vv^*$, we have
$$\tr(LS)=v^*Lv=2\ell'\ell\ell_{12}+\ell^2\ell_{22}=\ell^2(\ell_{22}+2\ell')=-\ell^2.$$

\smallskip
{\bf Case 2.} Case 1 does not apply but there are $k\ge 2$ and
$u_1,\dots,u_k\in\R^m\setminus\{0\}$ such that $\sum_{i=1}^ku_i^*Lu_i=0$.

Here we set $\ell:=u_1^*Lu_1\neq 0$ and write $-\ell=\ell_1^2-\ell_2^2$ where
$\ell_1:=\frac12(-\ell+1)\in\rxlin$ and $\ell_2:=\frac12(-\ell-1)\in\rxlin$.
Then we can use the quadratic sos-matrix
$$S:=\ell_1^2u_1u_1^*+\ell_2^2\sum_{i=2}^ku_iu_i^*=\ell_1^2u_1u_1^*-\ell_2^2u_1u_1^*=-\ell u_1u_1^*$$
to get $\tr(LS)=-\ell u_1^*Lu_1=-\ell^2$.
\end{proof}

The certificate \eqref{eq:lowdim} of low-dimensionality exists for \emph{some} but in general \emph{not for every} affine hyperplane containing the
spectrahedron. We illustrate this by the following example where the spectrahedron has codimension two and therefore is contained in
infinitely many affine hyperplanes \emph{only one} of which allows for a certificate \eqref{eq:lowdim}.

\begin{example}\label{ex:new}
Let
\bes
\begin{split}
L & =\begin{bmatrix} 0 & X_1 & 0 \\ X_1 & X_2 & X_3\\0 & X_3 & X_1\end{bmatrix}.
\end{split}
\ees 
Then $S_{L}=\{ (0,x_2,0)\in\R^3 \mid x_2\geq0\}$ and the (affine) hyperplanes containing $S_L$ are $\{X_1=0\}$ and $\{aX_1+X_3=0\}$ ($a\in\R$).
As is shown in Case 1 of the proof of Proposition \ref{exsos}, the certificate of low-dimensionality \eqref{eq:lowdim} exists for the hyperplane
$\{X_1=0\}$, i.e., there is a quadratic sos-matrix $S$ such that $-X_1^2=\tr(LS)$. However, none of the other hyperplanes containing $S_L$ allows
for the certificate \eqref{eq:lowdim}.

Otherwise assume that there is $a\in\R$ such that $\{aX_1+X_3=0\}$ has a also corresponding certificate. Combining it with the one
for $\{X_1=0\}$, we get a quadratic sos-matrix $S$ such that
$$-(2a^2)X_1^2-2(aX_1+X_3)^2=\tr(LS)$$
which implies
$$-X_3^2=(2aX_1+X_3)^2+(-(2a^2)X_1^2-2(aX_1+X_3)^2)\in M_L^{(1)}.$$
Specializing $X_3$ to $1$, one gets the contradiction $-1\in M_{L'}^{(1)}$ where $L'$ is the linear pencil from
Example~\ref{ex:sturm2}.
\end{example}

\subsection{Iterating the search for an affine hyperplane} We now carry out the slightly technical but easy iteration of Proposition \ref{exsos} announced
in Subsection \ref{subs:lowdim} and combine
it with Proposition \ref{usual}. We get a new type of Positivstellensatz for linear polynomials on spectrahedra with bounded degree complexity.

\begin{theorem}[Positivstellensatz for linear polynomials on spectrahedra]\label{indsos}\mbox{}\\
Let $L\in S\rx^{m\times m}$ be a linear pencil and $f\in\rx_1$. Then $$f\ge0\text{ on }S_L$$ if and only if  there exist
$\ell_1,\dots,\ell_n\in\rx_1$, quadratic sos-matrices $S_1,\dots,S_n\in S\rx^{m\times m}$, a matrix $S\in\smpsd$
and $c\ge0$ such that
\begin{align}
\ell_i^2+\tr(LS_i)&\in(\ell_1,\dots,\ell_{i-1})\quad\text{for}\quad i\in\{1,\dots,n\},\qquad\text{and}\label{indsos1}\\
f-c-\tr(LS)&\in(\ell_1,\dots,\ell_n).\label{indsos2}
\end{align}
\end{theorem}

\begin{proof}
We first prove that $f\ge0$ on $S_L$ in the presence of \eqref{indsos1} and \eqref{indsos2}.

The traces in \eqref{indsos1} and \eqref{indsos2} are elements of $M_L$ and therefore nonnegative on $S_L$.
Hence it is clear that constraint \eqref{indsos2} gives $f\ge0$ on $S_L$ if we show that $\ell_i$ vanishes on $S_L$
for all $i\in\{1,\dots,n\}$. Fix $i\in\{1,\dots,n\}$ and assume by induction that $\ell_1,\dots,\ell_{i-1}$ vanish on $S_L$. Then
\eqref{indsos1} implies $\ell_i^2+\tr(LS_i)$ vanishes on $S_L$ and therefore also $\ell_i$.

Conversely, suppose now that $f\ge0$ on $S_L$. We will obtain the data with properties \eqref{indsos1} and \eqref{indsos2}
by induction on the number of variables $n\in\N_0$.

To do the induction basis, suppose first that $n=0$. Then $f\ge0$ on $S_L$ just means that the real number $f$ is
nonnegative if $L\in\sm$ is positive semidefinite. But if $f\ge0$, then it suffices to choose $c:=f\ge0$ and $S:=0$
to obtain \eqref{indsos2} with $n=0$, and the other condition \eqref{indsos1} is empty since $n=0$.
We now assume that $f<0$ and therefore $L\not\succeq0$. Then we choose $u\in\R^m$ with
$u^*Lu=f$. Setting $S:=uu^*\in\smpsd$ and $c:=0$, we have $$f-c-\tr(LS)=f-u^*Lu=f-f=0,$$ as required.

For the induction step, we now suppose that $n\in\N$ and that we know already how to find the required data for linear
pencils in $n-1$ variables. We distinguish two cases and will use the induction hypothesis only in the second one.

{\bf Case 1.} $S_L$ contains an interior point.

In this case, we set all $\ell_i$ and $S_i$ to zero so that \eqref{indsos1} is trivially satisfied. Property
\eqref{indsos2} can be fulfilled by Proposition \ref{usual}.

{\bf Case 2.} The interior of $S_L$ is empty.

In this case, we apply
Proposition~\ref{exsos} to obtain $0\neq\ell_1\in\rx_1$ and a quadratic sos-matrix $S_1\in S\rx^{m\times m}$ with
\begin{equation}\label{anfang}
\ell_1^2+\tr(LS_1)=0.
\end{equation}
The case where $\ell_1$ is constant is trivial. In fact, in this case we can choose all remaining data being zero since
$(\ell_1,\dots,\ell_i)=(\ell_1)=\rx$ for all $i\in\{1,\dots,n\}$.

From now on we therefore assume $\ell_1$ to be non-constant. But then the reader easily checks that there is no harm
carrying out an affine linear variable transformation which allows us to assume $\ell_1=X_n$.
We then apply the induction hypothesis to the linear pencil $L':=L(X_1,\dots,X_{n-1},0)$ and the linear polynomial
$f':=f(X_1,\dots,X_{n-1},0)$ in $n-1$ variables to obtain
$\ell_2,\dots,\ell_n\in\rx$, quadratic sos-matrices $S_2,\dots,S_n\in S\rx^{m\times m}$, a matrix $S\in\smpsd$ and
a constant $c\ge0$ such that
\begin{align}
\ell_i^2+\tr(L'S_i)&\in(\ell_2,\dots,\ell_{i-1})\quad\text{for}\quad i\in\{2,\dots,n\}\qquad\text{and}\label{isos1}\\
f'-c-\tr(L'S)&\in(\ell_2,\dots,\ell_n).\label{isos2}
\end{align}
Noting that both $f-f'$ and $\tr(LS_i)-\tr(L'S_i)=\tr((L-L')S_i)$ are contained in the ideal $(X_n)=(\ell_1)$,
we see that \eqref{isos1} together with \eqref{anfang}
implies \eqref{indsos1}. In the same manner, \eqref{isos2} yields \eqref{indsos2}.
\end{proof}

\subsection{The real radical and Schur complements}\label{subs:schur}

Let $L\in S\rx^{m\times m}$ be a linear pencil and $q\in\rx$ a (quadratic) sos-polynomial such that
$-q=\tr(LS)$ for some (quadratic) sos-matrix $S$ like in \eqref{eq:lowdim} above. In order to resolve the third issue mentioned
in Subsection \ref{subs:lowdim}, we would like to get our hands on (cubic) polynomials vanishing on $\{q=0\}$. In other words, we want to
implement the ideals appearing in \eqref{indsos1} and \eqref{indsos2} in an SDP.

\def\cI{\mathcal I}

Recall that for any ideal $\cI\subseteq\rx$, its \emph{radical} $\sqrt\cI$ and its \emph{real radical} $\rrad\cI$ are the ideals defined by
\begin{align*}
\rad\cI&:=\{p\in\rx\mid \exists k\in\N:f^k\in\cI\}\qquad\text{and}\\
\rrad\cI&:=\{p\in\rx\mid \exists k\in\N:\exists s\in\sum\rx^2:f^{2k}+s\in\cI\}.
\end{align*}
An ideal $\cI\subseteq\rx$ is called \emph{radical} if $\cI=\rad\cI$ and \emph{real radical} if $\cI=\rrad\cI$.

By the Real Nullstellensatz \cite{bcr,ma,pd}, each polynomial vanishing on the real zero set $\{q=0\}$ of $q$
lies in $\rrad{(q)}$. This gives a strategy of how to find the cubic polynomials vanishing on
$\{q=0\}$, cf. Proposition \ref{matrad} and Lemma \ref{vepaco} below. The Real Nullstellensatz plays only a motivating role for us.
We will only use its trivial converse: Each element of $\rrad{(q)}$ vanishes on $\{q=0\}$.

The question is now how to model the search for elements in the real radical ideal by SDP. The key to this will be to represent polynomials
by matrices as it is done in the Gram matrix method mentioned at the beginning of Section \ref{sec:ram}. For this we have to introduce notation.

For each $d\in\N_0$, let $\s d:=\dim\rx_d=\binom{d+n}n$ denote the number of monomials of degree at most $d$ in $n$ variables
and $\vecx d\in\rx^{\s d}$ the column vector
$$\vecx d:=\begin{bmatrix}1&X_1&X_2&\dots&X_n&X_1^2&X_1X_2&\dots&\dots&X_n^d\end{bmatrix}^*$$
consisting of these monomials ordered first with respect to the degree and then lexicographic.

The following proposition shows how to find elements of degree at most $d+e$ (represented by a matrix $W$) in the real radical
$\cI:=\rrad{(q)}$ of the ideal generated by a polynomial $q\in\rx_{2d}$
(represented by a symmetric matrix $U$, i.e., $q=\vecx d^*U\vecx d$). We will later use it with $d=1$ and $e=2$ since $q$ will be
quadratic and we will need cubic elements in $\cI$.
Note that $$U\succeq W^*W\iff\begin{bmatrix}I&W\\W^*&U\end{bmatrix}\succeq0$$
by the method of Schur complements.

\begin{proposition}\label{matrad}
Let $d,e\in\N_0$,  $\cI$ a real radical ideal of $\rx$ and $U\in S\R^{s(d)\times s(d)}$ such that
$\vecx d^*U\vecx d\in\cI$.  Suppose $W\in\R^{s(e)\times s(d)}$ with $U\succeq W^*W$. Then $\vecx e^*W\vecx d\in\cI$.
\end{proposition}

\begin{proof}
Since $U-W^*W$ is positive semidefinite, we find $B\in\R^{s(d)\times s(d)}$ with
$U-W^*W=B^*B$.
Now let $p_i\in\rx$ denote the $i$-th entry of $W\vecx d$ and $q_j$ the $j$-th entry of $B\vecx d$.
From
\begin{align*}
p_1^2+\dots+p_{s(e)}^2+q_1^2+\dots+q_{s(d)}^2&=(W\vecx d)^*W\vecx d+(B\vecx d)^*B\vecx d\\
&=\vecx d^*(W^*W+B^*B)\vecx d=\vecx d^*U\vecx d\in\cI
\end{align*}
it follows that $p_1,\dots,p_{s(e)}\in\cI$ since $\cI$ is real radical. Now
$$\vecx e^*W\vecx d=\vecx e^*[p_1\dots p_{s(e)}]^*=[p_1\dots p_{s(e)}]\vecx e\in\cI$$
since $\cI$ is an ideal.
\end{proof}

The following lemma is a weak converse to Proposition~\ref{matrad}. Its proof relies heavily
on the fact that only linear and quadratic polynomials are involved.

\begin{lemma}\label{vepaco}
Suppose $\ell_1,\dots,\ell_t\in\rxlin$ and $q_1,\dots,q_t\in\rx_2$. Let $U\in S\R^{s(1)\times s(1)}$ be such that
$$\vecx 1^*U\vecx 1=\ell_1^2+\dots+\ell_t^2.$$
Then there exists $\la>0$ and $W\in\R^{s(2)\times s(1)}$ such that $\la U\succeq W^*W$ and
$$\vecx 2^*W\vecx 1=\ell_1q_1+\dots+\ell_t q_t.$$
\end{lemma}

\begin{proof}
Suppose that at least one $q_i\neq0$ (otherwise take $W=0$).
Choose column vectors $c_i\in\R^{s(2)}$ such that $c_i^*\vecx2=q_i$.
Now let $W\in\R^{s(2)\times s(1)}$ be the matrix defined by $W\vecx1=\sum_{i=1}^t\ell_ic_i$ so that
$$\vecx2^*W\vecx1=\sum_{i=1}^t\ell_i\vecx2^*c_i=\sum_{i=1}^t\ell_ic_i^*\vecx2=
\sum_{i=1}^t\ell_iq_i.$$
Moreover, we get
$$\vecx1^*W^*W\vecx1=(W\vecx1)^*W\vecx1=\sum_{i,j=1}^t(\ell_ic_i)^*(\ell_jc_j)$$
and therefore for all $x\in\R^n$
\begin{align*}
\begin{bmatrix}1&x_1&\dots&x_n\end{bmatrix}W^*W\begin{bmatrix}1\\x_1\\\vdots\\x_n\end{bmatrix}
&=\sum_{i,j=1}^t(\ell_i(x)c_i)^*(\ell_j(x)c_j)\\
&\le\frac12\sum_{i,j=1}^t((\ell_i(x)c_i)^*(\ell_i(x)c_i)+(\ell_j(x)c_j)^*(\ell_j(x)c_j))\\
&\le t\sum_{i=1}^t(\ell_i(x)c_i)^*(\ell_i(x)c_i)\le\la\sum_{i=1}^t\ell_i(x)^2,
\end{align*}
where we set $\la:=t\sum_{i=1}^tc_i^*c_i>0$. Therefore
$$\begin{bmatrix}1&x_1&\dots&x_n\end{bmatrix}(\la U-W^*W)\begin{bmatrix}1&x_1&\dots&x_n\end{bmatrix}^*\ge0$$
for all $x\in\R^n$. By homogeneity and continuity this implies $y^*(\la U-W^*W)y\ge0$ for all
$y\in\R^{s(1)}$, i.e., $\la U\succeq W^*W$.
\end{proof}

\subsection{The sums of squares dual of an SDP}

Given an SDP of the form (P) described in Subsection~\ref{subs:standard}, the following is what we call its
\emph{sums of squares dual}:
$$
\begin{array}[t]{rll}
\text{maximize}&a\\
\text{subject to}&S\in\sm,\ S\succeq0,\ a\in\R\\
&S_1,\dots,S_n\in S\rx^{m\times m}\text{ quadratic sos-matrices}\\
&U_1,\dots,U_n\in S\R^{s(1)\times s(1)}\\
&W_1,\dots,W_n\in\R^{s(2)\times s(1)}\\
&\vecx1^*U_i\vecx1+\vecx2^*W_{i-1}\vecx1+\tr(LS_i)=0&(i\in\{1,\dots,n\})\\
&U_i\succeq W_i^*W_i&(i\in\{1,\dots,n\})\\
&\ell-a+\vecx2^*W_n\vecx1-\tr(LS)=0.
\end{array}
$$
Just like Ramana's extended Lagrange-Slater dual \cite{ra} it can be written down in polynomial time (and hence has polynomial size)
in the bit size of the primal (assuming the latter has rational coefficients) and it guarantees that strong duality (i.e., weak duality,
zero gap and dual attainment, see Subsection \ref{subs:standard}) always holds.
Similarly, the facial reduction \cite{bw,TW} gives rise to a good duality theory of SDP. We refer the reader to
\cite{Pat} for a unified treatment of these two constructions.

As mentioned at the beginning of Section \ref{sec:ram}, the quadratic sos-matrices can easily be modeled by SDP constraints using the
Gram matrix method, and the polynomial identities can be written as linear equations by comparing coefficients.

The $S_i$ serve to produce negated quadratic sos-polynomials vanishing on $S_L$ (cf. Proposition \ref{exsos}) which are captured by the
matrices $U_i$. From this, cubics vanishing on $S_L$ are produced (cf. Subsection \ref{subs:schur}) and represented by the matrices
$W_i$. These cubics serve to implement the congruence modulo the ideals from \eqref{indsos1} and \eqref{indsos2}. The whole procedure
is iterated $n$ times.

Just as Proposition \ref{usual} corresponds to the standard SDP duality, Theorem \ref{sosram} translates into the strong duality for the sums of squares
dual. Before we come to it, we need a folk lemma which is well-known (e.g.~from the theory of Gröbner bases) and which we prove
for the convenience of the reader.

\begin{lemma}\label{folk}
Suppose $d\in\N$, $f\in\rx_d$ and $\ell_1,\dots,\ell_t\in\rx_1$ are non-constant linear polynomials such that
$f\in(\ell_1,\dots,\ell_t)$. Then there exist
$p_1,\dots,p_t\in\rx_{d-1}$ such that $f=p_1\ell_1+\dots+p_t\ell_t$.
\end{lemma}

\begin{proof} We proceed by induction on $t\in\N_0$. For $t=0$, there is nothing to show.
Now let $t\in\N$ and suppose the lemma is already proved with $t$ replaced by $t-1$.
We may assume that $\ell_1=X_1-\ell$ with
$\ell\in\R[X_2,\dots,X_n]$ (otherwise permute the variables appropriately and scale~$\ell_1$). Write
$f=\sum_{|\al|\le d}a_\al\x^\al$ with $a_\al\in\R$. Setting $g:=f(\ell,X_2,\dots,X_n)$, we have
$$f-g=\sum_{\genfrac{}{}{0pt}{1}{|\al|\le d}{1\le\al_1}}
a_\al(X_1^{\al_1}-\ell^{\al_1})X_2^{\al_2}\dotsm X_n^{\al_n}=p_1(X_1-\ell)=p_1\ell_1$$
where $p_1:=\sum_{\genfrac{}{}{0pt}{1}{|\al|\le d}{1\le\al_1}}
a_\al\left(\sum_{i=0}^{\al_1-1}X_1^i\ell^{\al_1-1-i}\right)X_2^{\al_2}\dotsm X_n^{\al_n}\in\rx_{d-1}$.
Moreover, $g\in(\ell_2,\dots,\ell_t)$ and therefore $g=p_2\ell_2+\dots+p_t\ell_t$ for some $p_2,\dots,p_t\in\rx_{d-1}$ by
induction hypothesis. Now 
\[
f=(f-g)+g=p_1\ell_1+\dots+p_t\ell_t.
\qedhere\]
\end{proof}

\begin{theorem}[Sums of squares SDP duality]\label{sosram}
Let $L\in S\rx^{m\times m}$ be a linear pencil and $f\in\rx_1$. Then $$f\ge0\text{ on }S_L$$ if and only if there exist
quadratic sos-matrices
$S_1,\dots,S_n\in S\rx^{m\times m}$, matrices $U_1,\dots,U_n\in S\R^{s(1)\times s(1)}$, $W_1,\dots,W_n\in\R^{s(2)\times s(1)}$,
$S\in\smpsd$ and $c\in\R_{\ge0}$ such that
\begin{align}
&\vecx1^*U_i\vecx1+\vecx2^*W_{i-1}\vecx1+\tr(LS_i)=0&(i\in\{1,\dots,n\}),\label{sosram1}\\
&U_i\succeq W_i^*W_i&(i\in\{1,\dots,n\}),\label{sosram2}\\
&f-c+\vecx2^*W_n\vecx1-\tr(LS)=0,\label{sosram3}
\end{align}
where $W_0:=0$.
\end{theorem}

\begin{proof} We first prove that existence of the above data implies $f\ge0$ on $S_L$.
All we will use about the traces appearing in \eqref{sosram1} and \eqref{sosram3} is that they are polynomials nonnegative on $S_L$.
Let $\cI$ denote the real radical ideal of all polynomials vanishing on $S_L$.
It is clear that \eqref{sosram3} gives $f\ge0$ on $S_L$ if we show that
$\vecx2W_n\vecx1\in\cI$. In fact, we prove by induction that $\vecx2^*W_i\vecx1\in\cI$
for all $i\in\{0,\dots,n\}$.

The case $i=0$ is trivial since $W_0=0$ by definition. Let $i\in\{1,\dots,n\}$ be given and suppose that $\vecx2^*W_{i-1}\vecx1\in\cI$.
Then \eqref{sosram1} shows $\vecx1^*U_i\vecx1\le0$ on $S_L$. On the other hand \eqref{sosram2} implies in particular
$U_i\succeq0$ and therefore $\vecx1^*U_i\vecx1\ge0$ on $S_L$. Combining both, $\vecx1^*U_i\vecx1\in\cI$.
Now Proposition \ref{matrad}  implies $\vecx2^*W_i\vecx1\in\cI$ by \eqref{sosram2}.
This ends the induction and shows $f\ge0$ on $S_L$ as claimed.

Conversely, suppose now that $L$ is infeasible. By Theorem \ref{indsos} and Lemma \ref{folk}, we can choose
$\ell_1,\dots,\ell_n\in\rx_1$, quadratic sos-matrices $S_1',\dots,S_n'\in S\rx^{m\times m}$, $S\in\smpsd$ and
$q_{ij}\in\rx_2$ ($1\le j\le i\le n$) such that
\begin{align}
\ell_1^2+\dots+\ell_i^2+\tr(LS_i')&=\sum_{j=1}^{i-1}q_{(i-1)j}\ell_j&(i\in\{1,\dots,n\})\qquad\text{and}
\label{prep1}\\
f-c-\tr(LS)&=\sum_{j=1}^nq_{nj}\ell_j.\label{prep2}
\end{align}
There are two little arguments involved in this: First, \eqref{indsos1} can trivially be rewritten as
$\ell_1^2+\dots+\ell_i^2+\tr(LS_i')\in(\ell_1,\dots,\ell_{i-1})$ for $i\in\{1,\dots,n\}$. Second, in Lemma \ref{folk} the $\ell_i$ are
assumed to be non-constant but can be allowed to equal zero. But if some $\ell_i\neq0$ is constant, then we may set
$\ell_{i+1}=\dots=\ell_n=0$ and $S_{i+1}'=\dots=S_n'=S=0$.

Now define $U_i'\in S\R^{s(1)\times s(1)}$ by $\vecx1^*U_i'\vecx1=\ell_1^2+\dots+\ell_i^2$ for $i\in\{1,\dots,n\}$. Using Lemma \ref{vepaco},
we can then choose $\la>0$ and $W_1',\dots,W_n'\in\R^{s(2)\times s(1)}$ such that
\begin{equation}\label{cop}
\la U_i'\succeq W_i'^*W_i'
\end{equation}
and
$\vecx2^*W_i'\vecx1=-\sum_{j=1}^iq_{ij}\ell_j$ for $i\in\{1,\dots,n\}$. Setting $W_n:=W_n'$, equation \eqref{prep2}
becomes \eqref{sosram3}. Moreover, \eqref{prep1} can be rewritten as
\begin{equation}\label{rew}
\vecx1^*U_i'\vecx1+\vecx2^*W_{i-1}'\vecx1+\tr(LS_i')=0\qquad(i\in\{1,\dots,n\})
\end{equation}
To cope with the problem that $\la$ might be larger than $1$ in \eqref{cop}, we look for $\la_1,\dots,\la_n\in\R\pos$ such
that $U_i\succeq W_i^*W_i$ for all $i\in\{1,\dots,n\}$ if we define $U_i:=\la_iU_i'$ and $W_{i-1}:=\la_iW_{i-1}'$ for all
$i\in\{1,\dots,n\}$ (in particular $W_0=W_0'=0$). With this choice, the desired linear matrix inequality \eqref{sosram2} is now
equivalent to $\la_iU_i'\succeq\la_{i+1}^2W_i'^*W_i'$ for $i\in\{1,\dots,n-1\}$ and $\la_nU_n'\succeq W_n'^2$.
Looking at \eqref{cop}, we therefore see that any choice of the $\la_i$ satisfying $\la_i\ge\la\la_{i+1}^2$ for
$i\in\{1,\dots,n-1\}$ and $\la_n\ge\la$ ensures \eqref{sosram2}. Such a choice is clearly possible. Finally, equation
\eqref{rew} multiplied by $\la_i$ yields \eqref{sosram1} by setting $S_i:=\la_iS_i'$ for $i\in\{1,\dots,n\}$.
\end{proof}

\subsection{The real radical and the quadratic module}
Let $L$ be a linear pencil. In Definition \ref{def:scm}, we have introduced the convex cone $C_L\subseteq\rx_1$ and the quadratic module
$M_L\subseteq\rx$ associated to $L$  consisting of polynomials which are obviously nonnegative of the spectrahedron $S_L$.
In Theorem \ref{thm:main1}, we have shown that the quadratic module $M_L$ can be used to certify infeasibility of $L$, in the sense that
$S_L=\emptyset$ implies $-1\in M_L$. On the contrary, we have seen in Subsection \ref{subs:ex}, that the convex cone $C_L$ is in general too
small to detect infeasibility of $L$ in this way.

In this subsection, we turn over to the more general question of certifying nonnegativity of \emph{arbitrary} linear polynomials on $S_L$
(as opposed to just the constant polynomial $-1$).
The following example shows that $M_L$ (and henceforth its subset $C_L$) does not, in general, contain all linear polynomials nonnegative
on $S_L$.

\begin{example}\label{new2}
Consider $$L= \begin{bmatrix} 1 & X \\ X & 0\end{bmatrix}.$$
Then $S_L=\{0\}$. Hence obviously $X\geq0$ on $S_L$. But it is easy to see that $X\not\in M_L$ \cite[Example 2]{za}.
\end{example}

Despite this example,
Theorem \ref{indsos} and its SDP-implementable version Theorem \ref{sosram} \emph{do} on the other hand yield algebraic certificates
of linear polynomials nonnegative on $S_L$. These two theorems have the advantage of being very well-behaved with respect to complexity
issues but have the drawback of their statements being somewhat technical. Leaving complexity issues aside, one can however come back to a nice
algebraic characterization of linear polynomials nonnegative on $S_L$.

Given $f\in\rx_1$ with $f\ge0$ on $S_L$, the certificates in Theorems \ref{indsos} and \ref{sosram} for being nonnegative on $S_L$ can actually
be interpreted as certificates of $f$ lying in the convex cone $C_L+\rad{\supp M_L}$ by means of Prestel's theory of semiorderings.
Note that each element of $C_L+\rad{\supp M_L}$ is of course nonnegative $S_L$ since the elements of $\rad{\supp M_L}$ vanish on $S_L$.

Finally, this will allow us to come back to the quadratic module $M_L$. We will show that it contains each linear polynomial nonnegative
on $S_L$ \emph{after adding an arbitrarily small positive constant}, see Corollary \ref{cor:almostsat}.

In this subsection, basic familiarity with real algebraic geometry as presented e.g.~in \cite{bcr,ma,pd} is needed.
The following proposition follows easily from Prestel's theory of semiorderings on a commutative ring, see for example \cite[1.4.6.1]{sc}.

\begin{proposition}\label{prestel}
Let $M$ be a quadratic module in $\rx$. Then
$$\rad{\supp M}=\rrad{\supp M}=\bigcap\{\supp S\mid \text{$S$ semiordering of $\rx$}, M\subseteq S\}.$$
\end{proposition}

We explicitly extract the following consequence since this is exactly what is needed in the sequel.

\begin{lemma}
Let $M$ be a quadratic module in $\rx$. Then
\begin{equation}
(\rad{\supp M}-M)\cap M\subseteq\rad{\supp M}.\label{prestelcor1}
\end{equation}
\end{lemma}

\begin{proof}
To prove \eqref{prestelcor1},
suppose $p\in M$ can be written $p=g-q$ with $g\in\rad{\supp M}$ and $q\in M$. By Proposition \ref{prestel}, we have
to show that $p\in\supp S$ for each semiordering $S$ of $\rx$ with $M\subseteq S$. But if such $S$ is given, then
$g\in\supp S$ and therefore $p=g-q\in -S$ as well as $p\in M\subseteq S$. Hence $p\in\supp S$.
\end{proof}

Having this lemma at hand, we can now give a conceptual interpretation of the certificates appearing in Theorem \ref{indsos}, disregarding the 
complexity of the certificate.

\begin{proposition}\label{indsos++}
If $L\in S\rx^{m\times m}$ is a linear pencil,
$f,\ell_1,\dots,\ell_n\in\rx_1$, $S_1,\dots,S_n\in S\rx^{m\times m}$ are quadratic sos-matrices, $S\in\smpsd$
and $c\in\R_{\ge0}$ such that \eqref{indsos1} and \eqref{indsos2} hold, then $f\in C_L+\rad{\supp M_L}$.
\end{proposition}

\begin{proof}
Set $\cI:=\rad{\supp M_L}$.
It is clear that \eqref{indsos2} gives $f\in C_L+\cI$ if we prove that
$\ell_i\in\cI$ for all $i\in\{1,\dots,n\}$. Fix $i\in\{1,\dots,n\}$ and assume by induction that $\ell_1,\dots,\ell_{i-1}\in\cI$. Then
\eqref{indsos1} implies $\ell_i^2+\tr(LS_i)\in\cI$ and therefore $\ell_i^2\in(\cI-M_L)\cap\sos\subseteq(\cI-M_L)\cap M_L
\subseteq\cI$ by \eqref{prestelcor1}.
\end{proof}

We get the same interpretation for the certificates from Theorem \ref{sosram}.

\begin{proposition}\label{sosram++}
If $L\in S\rx^{m\times m}$ is a linear pencil,
$f\in\rx_1$, $S_1,\dots,S_n\in S\rx^{m\times m}$ are quadratic sos-matrices,
$U_1,\dots,U_n\in S\R^{s(1)\times s(1)}$, $W_1,\dots,W_n\in\R^{s(2)\times s(1)}$, $S\in S\R^{m\times m}_{\succeq0}$ and $c\in\R_{\ge0}$ such that
\eqref{sosram1}, \eqref{sosram2} and \eqref{sosram2} hold, then $f\in C_L+\rad{\supp M_L}$.
\end{proposition}

\begin{proof}
Set $\cI:=\rad{\supp M_L}$.
It is clear that constraint \eqref{sosram3} gives $f\in C_L+\cI$ if we show that
$\vecx2W_n\vecx1\in\cI$. In fact, we show by induction that $\vecx2W_i\vecx1\in\cI$ for all $i\in\{0,\dots,n\}$.

The case $i=0$ is trivial since $W_0=0$ by definition. Let $i\in\{1,\dots,n\}$ be given and suppose
that we know already $\vecx2^*W_{i-1}\vecx1\in\cI$.
Then \eqref{sosram1} shows $\vecx1^*U_i\vecx1\in\cI-M_L$. On the other hand \eqref{sosram2} implies in particular
$U_i\succeq0$ and therefore $\vecx1^*U_i\vecx1\in\sos\subseteq M_L$. But then
$\vecx1^*U_i\vecx1\in(\cI-M_L)\cap M_L\subseteq\cI$ by \eqref{prestelcor1}. Now \eqref{sosram2} yields
$\vecx2^*W_i\vecx1\in\cI$ by Proposition \ref{matrad} since $\cI$ is real radical by Proposition \ref{prestel}.
This ends the induction.
\end{proof}

The following corollary is now a generalization of Proposition \ref{usual} working also for low-dimensional $S_L$ (note that $\supp M_L=(0)$ if
$S_L$ has non-empty interior).

\begin{corollary}\label{cor:radlinsat}
Let $L\in S\rx^{m\times m}$ be a linear pencil. Then
$$f\ge0\text{\ on\ }S_L\ \iff\ f\in C_L+\sqrt{\supp M_L}$$
for all $f\in\rx_1$.
\end{corollary}

\begin{proof}
Combine either Theorem \ref{indsos} with Proposition \ref{indsos++}, or Theorem \ref{sosram} with Proposition \ref{sosram++}.
\end{proof}

Now we come back to the quadratic module where we cannot avoid to add $\ep>0$ in order to get a certificate as was shown in
Example \ref{new2}.

\begin{corollary}\label{cor:almostsat}
Let $L\in S\rx^{m\times m}$ be a linear pencil. Then
$$f\ge0\text{\ on\ }S_L\ \iff\ \forall\ep>0:f+\ep\in M_L$$
for all $f\in\rx_1$.
\end{corollary}

\begin{proof}
To prove the non-trivial implication, let $f\in\rx_1$ with $f\ge0$ on $S_L$ be given. It suffices to show $f+\ep\in M_L$ for the special case $\ep=1$
(otherwise replace $f$ by $\ep f$ and divide by $\ep$). By Corollary \ref{cor:radlinsat}, there exists $g\in C_L$, $p\in\rx$ and $k\in\N$
such that $f=g+p$ and $p^k\in\cI:=\supp M_L$. Now $f+\ep=f+1=g+(f-g)+1=g+(p+1)$ and it is enough to show that $p+1\in M_L$. This will follow
from the fact that the image of $p+1$ is a square in the quotient ring $\rx/\cI$. Indeed since the image of $p$ in $\rx/\cI$ is nilpotent (in fact the image of
$p^k$ is zero), we can simply write down a square root of this element using the \emph{finite} Taylor expansion at $1$ of the square root function
in $1$ given by the binomial series:
\[p+1\equiv\left(\sum_{i=0}^{k-1}\binom{\frac12}ip^i\right)^2\qquad\text{mod }\cI.\qedhere\]
\end{proof}

\section{Matricial spectrahedra and complete positivity}\label{sec:hkm}

In this section we revisit some of the main results from 
\cite{hkm}, where we considered
{\em noncommutative $($matricial$)$ relaxations} of linear matrix inequalities
under the assumption of strict feasibility.
The purpose of this section is twofold.
First, in Subsection \ref{sec:tau} 
we explain which of the results from \cite{hkm}
generalize to {\em weakly feasible} linear matrix inequalities.
Second, in Subsection \ref{subsec:bla} 
we explain how our results from Section 3 pertain
to (completely) positive maps under the absence of positive definite
elements.

\subsection{Matricial relaxations of LMIs}

Suppose
\bes%q\label{eq:pencil}
L=A_0+\sum_{i=1}^n X_iA_i
\ees
is a linear pencil of size $m$. 
  Given
  $Z=\mbox{col}(Z_1,\dots,Z_n):=\begin{bmatrix}Z_1\\\vdots\\Z_n\end{bmatrix}\in (\SRdd)^{n},$
the \emph{evaluation} $L(Z)$ is defined as
$$
  L(Z)=I_d\otimes A_0 + \sum_{i=1}^n Z_i \otimes A_i\in S\R^{dm\times dm}.
$$
The \emph{matricial spectrahedron} of a linear pencil $L$ 
is
$$
\cD_{L}:= \bigcup_{d\in\N}\{Z\in (\SRdd)^n \mid L(Z) \succeq 0\}.
$$
Let
$$
\cD_L(d)=\{ Z\in (\SRdd)^n\mid L(Z) \succeq 0\}=\cD_L\cap
(\SRdd)^n.
$$
for $d\in\N$. The set $\cD_L(1)\subseteq\R^n$ is the feasibility set of the linear
matrix inequality 
$L(x)\succeq0$ and coincides with $S_L$ as introduced above.
We call $\cD_L$ \emph{bounded} if there is an $N\in\N$ with operator norm $\|Z\|\leq N$ for all $Z\in \cD_L$.

\def\col{{\rm col}}
\begin{lem}
For matrices $Z_1,\ldots,Z_n$ of the same size, we have
\[
\| \col(Z_1,\ldots,Z_n) \|^2 = \Big\| \sum_{i=1}^n Z_i^*Z_i \Big\|.
\]
\end{lem}

\begin{proof}
Suppose $Z_j\in\R^{d\times d}$.
Expand $\col(Z_1,\ldots,Z_n)\in \R^{dn\times d}$
to a $dn\times dn$ matrix $\hat Z$ by adding zero columns.
Then
\bes
\begin{split}
\| \col(Z_1,\ldots,Z_n) \|^2 &= \|\hat Z\|^2 
= \| \hat Z^* \hat Z \| \\
& = 
 \Big\| 
\begin{bmatrix} 
\sum_i Z_i^*Z_i & 0  & \cdots & 0\\
0 & 0 & \cdots & 0 \\
\vdots & \vdots & \ddots & \vdots \\
0 & 0 & \cdots & 0
\end{bmatrix}
\Big\| 
= 
\big\| \sum_i Z_i^*Z_i \big\|.
\qedhere
\end{split}
\ees
\end{proof}

Given linear pencils $L_1$ and $L_2$,
\beq
\label{eq:penc12}
 L_j= A_{j,0}+\sum_{i=1}^n X_i A_{j,i}  \in S\R\ac_1^{m_j\times m_j} , \quad j=1,2,
\eeq
we shall be interested in  the following inclusion for
matricial spectrahedra:
\beq\label{eq:bind1}
\cD_{L_1}\subseteq\cD_{L_2}.
\eeq
In this case we say that $L_1$ \emph{matricially dominates} $L_2$.
If $L_2\in\R\cx_1$ then \eqref{eq:bind1} is equivalent to 
$S_{L_1}\subseteq S_{L_2}$:

\begin{prop}\label{prop:1pos}
Let $L$ be a linear pencil of size $m$ and let $\ell\in\R\cx_1$. Then 
$$
\ell|_{S_L}\geq 0 \quad\iff\quad
\ell|_{\cD_L}\succeq0.
$$
\end{prop}

\begin{proof}
The implication $\Leftarrow$ is obvious as $\cD_L\supseteq S_L$.
For the converse assume $\ell|_{\cD_L}\not\succeq0$ 
and choose $Z=(Z_1,\ldots,Z_n)\in\cD_L$ with $\ell(Z)\not\succeq0$. Let
$v$ be an eigenvector of $\ell(Z)$ with negative eigenvalue.
For 
$v^*Zv=(v^*Z_1v,\ldots,v^*Z_nv)\in S_L$ 
we have
$$L(v^*Zv)=(I_m\otimes v)^*L(Z)(I_m\otimes v)\succeq0,$$
and
$$
0 > v^* \ell(Z) v = \ell(v^*Zv),
$$
whence $\ell|_{S_L}\not\geq0$.
\end{proof}

\begin{cor}
Let $L$ be a linear pencil. Then
$$\cD_L \text{ is bounded}\quad\iff\quad
S_L=\cD_L(1) \text{ is bounded.}
$$
\end{cor}

\begin{cor}
For a linear pencil $L$, 
$\cD_L=\varnothing$ if and only if $S_L=\varnothing$.
\end{cor}

Since empty spectrahedra were thoroughly analyzed in  previous
sections, in the sequel we shall always {\it assume} $S_L\neq\varnothing$.
Moreover, we assume that there is $Z\in \cD_L$ with $L(Z)\neq0$.
Then, by compressing and translating,
we ensure 
\smallskip
\ben[\rm {({\bf Asmp})}]
\item
$0\neq A_0\succeq0$.
\een

\smallskip\noindent
If the interior of $S_L$ is non-empty, we could further reduce to 
strictly feasible linear pencils (cf.~\cite[Proposition 2.1]{hkm}), but this
is not the case we are interested in here.

\begin{lemma}
If $S_L$ is bounded, then $A_0,\ldots,A_n$ are linearly independent.
\end{lemma}

\begin{proof}
This is easy; or see \cite[Proposition 2.6(2)]{hkm}.
\end{proof}

\def\cS{\mathscr S}
We now introduce  subspaces to be used in our considerations:
\bes%\label{eq:opSys}
\cS_j = \Span \{A_{j,i}\mid i=0,\ldots,n\}
= \Span\{L_j(z)\mid z\in\R^n\}
\subseteq
S\R^{m_j\times m_j}.
\ees
We call $\cS_j=\cS (L_j)$ the (nonunital) \emph{operator system} associated
to the linear pencil $L_j$.
Conversely, to each linear subspace $\cS\subseteq S\R^{m\times m}$
we can associate a linear pencil $L$  with $\cS=\cS(L)$ by fixing a basis for
$\cS$. We can even enforce additional properties on $L$, see Proposition
\ref{prop:makeMeBounded} below.

The key tool in studying inclusions of matricial  spectrahedra
as in \eqref{eq:bind1} is the linear map
$\tau: \cS(L_1) \to \cS(L_2)$ we now define.

\begin{definition}\label{def:tau}
Let
$L_1,L_2$ be linear pencils as in \eqref{eq:penc12}.
 If
$A_{1,0},\dots,A_{1,n}$
are linearly independent
$($e.g.~$S_{L_1}$ is bounded$)$,
 we define the linear map
$$\tau : \cS_1 \to \cS_2 , \quad A_{1,i}\mapsto A_{2,i}\quad(i=0,\dots,n).$$
\end{definition}

We shall soon see that,
assuming \eqref{eq:bind1},  $\tau$ has a property called complete positivity, 
which we now introduce.
Let $\cS_j\subseteq S\R^{m_j\times m_j}$ be linear subspaces
and $\ph:\cS_1\to\cS_2$ a linear map.
For $d\in\N$, $\phi$ induces the map 
$$\ph_d=I_d\otimes\ph:\R^{d\times d}\otimes\cS_1=\cS_1^{d\times d}
 \to\cS_2^{d\times d},
\quad M\otimes A \mapsto M \otimes \ph(A),
$$
called an \emph{ampliation} of $\ph$.
Equivalently,
$$
\ph_d\left( \begin{bmatrix} T_{11} & \cdots & T_{1d} \\
\vdots & \ddots & \vdots \\
T_{d1} & \cdots & T_{dd}\end{bmatrix}\right)
=\begin{bmatrix} \ph(T_{11}) & \cdots & \ph(T_{1d}) \\
\vdots & \ddots & \vdots \\
\ph(T_{d1}) & \cdots & \ph(T_{dd})\end{bmatrix}
$$
for $T_{ij}\in\cS_1$.
We say that $\ph$ is \emph{$d$-positive} if 
$\ph_d$ is a positive map, i.e., maps positive semidefinite matrices
into positive semidefinite matrices. If $\ph$ is $d$-positive for every $d\in\N$,
then $\ph$ is \emph{completely positive}.

\begin{remark}
If $\cS_1$ does not contain positive semidefinite matrices,
then every linear map out of $\cS_1$ is completely positive.
On the other hand, 
under the presence of positive definite  elements 
there is a well established theory of completely positive maps, cf.~\cite{Pau}.
In the absence of positive definite elements 
in $\cS_1$ the structure of completely positive maps out of $\cS_1$ seems
to be more
complicated, and our results from Section \ref{sec:ram} describe
(completely) positive maps $\cS_1\to\R$ in this case; see 
Subsection \ref{subsec:bla} below.
\end{remark}

\subsection{The map
$\tau$ is completely positive: LMI matricial domination}
\label{sec:tau}

The main result of this section is the equivalence of
$d$-positivity of the map $\tau$ 
from Definition \ref{def:tau} 
with the inclusion
$\cD_{L_1}(d)\subseteq\cD_{L_2}(d)$. 
As a corollary, $L_1$ matricially dominates $L_2$ if and only
if $\tau$ is completely positive.

\begin{theorem}[cf.~\protect{\cite[Theorem 3.5]{hkm}}]
 \label{thm:tau}
Let
\[
  L_j=A_{j,0}+\sum_{i=1}^n X_i A_{j,i}
    \in S\R\ac_1^{d_j\times d_j}, \quad j=1,2
\]
be linear pencils $($with $0\neq A_{1,0}\succeq0)$,
 and assume
the spectrahedron $S_{L_1}$ is bounded.
Let $\tau$ be the linear map given in Definition $\ref{def:tau}$.
\ben[\rm (1)]
\item
$\tau$ is $d$-positive if and only if $\cD_{L_1}(d)\subseteq
\cD_{L_2}(d)$;
\item
$\tau$ is completely positive if and only if
$\cD_{L_1}\subseteq\cD_{L_2}$.
\een
\end{theorem}

We point out that in \cite{hkm} this theorem has been given 
under an additional assumption: the linear pencil $L_1$
was assumed to be strictly feasible.

To prove the theorem we start with a lemma.

\begin{lemma}\label{lem:bounded34}
Let $L=A_0+\sum_{i=1}^n X_iA_i$ be a linear pencil of size $m$ with $0\neq A_0\succeq0$ defining a bounded spectrahedron $S_L$.
Then:
\begin{enumerate}[\rm (1)]
\item
  if
 $\Lambda\in \R^{d\times d}$ and $Z\in (\SRdd)^n,$
  and if
\beq\label{eq:T1}
S:= \Lambda\otimes A_0 + 
\sum_{i=1}^nZ_i\otimes A_i
\eeq
  is symmetric, then  $\Lambda =\Lambda^*$;
\item if $S\succeq0$, then
$\Lambda\succeq 0$.
\end{enumerate}
\end{lemma}

\begin{proof}
 To prove item (1), suppose $S$ is symmetric. Then
$
0=S-S^* = (\Lambda-\Lambda^*)\otimes A_0.
$
Since $A_0\neq0$, $\Lambda=\Lambda^*$.

For (2),
  if $\Lambda \not \succeq 0$, then there is a unit vector $v$ such that
$v^*\La v < 0$.
 Consider the orthogonal projection $P\in S\R^{dm\times dm}$ from $\R^d\otimes\R^m$ onto $\R v \otimes \R^m$, and
let 
$Y=((v^*Z_iv)P_v)_{i=1}^n\in(S\R^{d\times d})^n$. Here $P_v\in S\R^{d\times d}$ is the orthogonal projection from $\R^d$ onto $\R v$. Note that
$P=P_v\otimes I_m$. Then the compression
$$
PSP= P( \Lambda\otimes A_0 +\sum_{i=1}^nZ_i\otimes A_i)P
     = (v^*\La v)P_v\otimes A_0  +\sum_{i=1}^nY_i\otimes A_i \succeq 0,
$$
which says that  $0\neq\sum_{i=1}^nY_i\otimes A_i\succeq 0$ since $0\neq A_0\succeq0$ and $v^*\La v<0$.
This implies $0\neq t Y\in \cD_L$ for all $t>0$;
contrary to $\cD_L$ being bounded.
\end{proof}

\begin{proof}[Proof of Theorem {\rm\ref{thm:tau}}]
In both of the statements, the direction $(\Rightarrow)$ is
obvious. We focus on the converses.

  Fix $d\in\N$. Suppose $S\in\cS_1^{d\times d}$ is positive semidefinite.
Then it is of the form \eqref{eq:T1} for some
$\Lambda\in\R^{d\times d}$ and $Z\in(\SRdd)^d$:
 \[
   S=\Lambda\otimes A_{1,0}  +\sum_{i=1}^n Z_i\otimes A_{1,i} 
         \succeq 0.
 \]
By Lemma \ref{lem:bounded34}, $\Ga\succeq0$.
If we replace $\Lambda$ by $\Lambda+\ep I$ for some $\ep>0$, the resulting
$S=S_\ep$ is still in $\cS_1^{d\times d}$, so without loss of generality
we may assume $\Lambda\succ0$.
   Hence,
 \[
   ( \Lambda^{-\frac12}\otimes I) S( \Lambda^{-\frac12}\otimes I) 
    =  I \otimes A_{1,0} +\sum_{i=1}^n(\Lambda^{-\frac12} Z_i
          \Lambda^{-\frac12}) \otimes A_{1,i}  \succeq 0.
 \]
  Since $\cD_1(d)\subseteq\cD_2(d)$, this implies
 \[
      I \otimes A_{2,0}+   \sum_{i=1}^n (\Lambda^{-\frac12} Z_i
          \Lambda^{-\frac12}) \otimes A_{2,i}\succeq 0.
 \]
   Multiplying on the left and right by $ \Lambda^{\frac12}\otimes I$
   shows
 \[
\tau(S_\ep)=
   \Lambda\otimes A_{2,0}+\sum_{i=1}^n   Z_i \otimes A_{2,i}         \succeq 0.
 \]
  An approximation argument now
  implies that if $S\succeq 0$, then $\tau(S)\succeq 0$ and
  hence $\tau$ is $d$-positive proving (1). Item (2) follows immediately.
\end{proof}

We next use the complete positivity of $\tau$ under the assumption
of LMI matricial domination to give an algebraic characterization of
linear pencils $L_1$ and $L_2$ producing an inclusion $\cD_{L_1}\subseteq\cD_{L_2}$.

\begin{corollary}[Linear Positivstellensatz, 
cf.~\protect{\cite[Theorem 1.1]{hkm}}]
 \label{cor:tau}
Let $L_1$ and $L_2$ be linear pencils of sizes $m_1$ and $m_2$, respectively. Assume that $S_{L_1}$ is bounded
and $\cS_1$ contains a positive definite matrix
$($e.g.~$L_1$ is strictly 
feasible$)$.
Then the following are equivalent:
\ben[\rm (i)]
\item
$\cD_{L_1}\subseteq\cD_{L_2}$;
\item
there exist $\mu\in\N$ and $V\in\R^{\mu m_1\times m_2}$ such that
\beq\label{eq:linPos}
L_2= V^* \big( I_\mu \otimes L_1 \big) V.
\eeq
\een
\end{corollary}
Before turning to the proof of the 
corollary, we pause for a remark.

\begin{remark}
Equation  \eqref{eq:linPos} can be equivalently written
as 
\beq\label{eq:linPss}
L_2=\sum_{j=1}^\mu V_j^* L_1V_j,
\eeq
where
$V_j\in\R^{m_1\times m_2}$ and
$V=
\text{col}(V_1,\ldots,V_\mu)$.
Moreover, $\mu$ can be uniformly bounded 
by Choi's characterization
\cite[Proposition 4.7]{Pau}
of completely positive maps between matrix algebras. In fact,
$\mu\leq m_1 m_2$.
\end{remark}

\begin{proof}[Proof of Corollary {\rm\ref{cor:tau}}]
The proof of the nontrivial implication (i) $\Rightarrow$ (ii) 
proceeds as follows. First invoke 
Arveson's extension theorem \cite[Theorem 7.5]{Pau} to extend
$\tau$ from Definition \ref{def:tau} to a completely positive map $\psi:\R^{m_1\times m_1}\to\R^{m_2\times m_2}$.
(\emph{Caution}: this is where
the existence of a positive definite matrix in  $\cS_1$ is used.)
Then 
apply the Stinespring
representation theorem \cite[Theorem 4.1]{Pau} to obtain
\beq\label{eq:taupi}
\psi(a)= V^* \pi(a) V,\quad a\in\R^{ m_1\times m_1}
\eeq
for some unital
$*$-representation $\pi:\R^{m_1\times m_1}\to\R^{m_3\times m_3}$
and $V:\R^{m_1}\to\R^{m_3}$.
By dissecting the proof (or see \cite[pp.~12-14]{hkm})
we see that $\pi$ is 
(unitarily equivalent to) a multiple of the identity representation, i.e.,
$\pi(a)= I_\mu\otimes a$ for some $\mu\in\N$ and all $a\in\R^{m_1\times m_1}$.
Hence
\eqref{eq:taupi} implies \eqref{eq:linPos}.
\end{proof}

\subsection{Examples and concluding remarks}\label{subsec:bla}

\begin{example}\label{ex:za}
Let us revisit Example \ref{new2} to show that the existence of a positive
definite matrix in $\cS_1$ is needed for Corollary \ref{cor:tau}
to hold.
Consider $$L_1= \begin{bmatrix} 1 & X \\ X & 0\end{bmatrix},
\quad L_2=X.$$
Then $S_{L_1}=\{0\}$ and $\cD_{L_1}=\{0\}$. 
Hence obviously $\cD_{L_1}\subseteq\cD_{L_2}$.
However, $L_2$ does not admit a representation of the form
\eqref{eq:linPss}.
Indeed, such a certificate is equivalent to $X\in C_{L_1}$; but
it was already pointed out that $X$ is not even a member of $M_{L_1}$.
\end{example}

We finish this paper by showing how our Theorem~\ref{indsos}
(cf. Theorem~\ref{sosram}) pertains to positive maps $\cS\to\R$ for a linear
subspace $\cS\subseteq S\R^{m\times m}$.
For this we start by associating to $\cS$ a linear pencil  $L$
with $\cS(L)=\cS$ and bounded $S_L$.

\begin{prop}\label{prop:makeMeBounded}
Let $\cS\subseteq S\R^{m\times m}$ be a linear subspace 
of dimension $n+1\geq2$
containing a nonzero positive semidefinite matrix. 
Then there exist $A_0,A_1,\ldots, A_n\in\cS$ such that
\ben[\rm (1)]
\item $A_0\succeq0$;
\item $\cS=\Span\{A_0,\ldots,A_n\}$;
\item 
for the linear pencil $L:=A_0+X_1 A_1+\cdots+X_n A_n$ the
spectrahedron $S_L$ is bounded.
\een
\end{prop}

\begin{proof}
We will use the trace inner product
\[
(A,B) \mapsto \langle A,B\rangle:= \tr(AB)
\]
on $S\R^{m\times m}$.

Let $A_0$ be a maximum rank positive semidefinite matrix in $\cS$.
Without loss of generality we may assume
\[
A_0= \begin{bmatrix} I_s & 0 \\ 0 & 0_{m-s} \end{bmatrix}
\]
for some $1\leq s\leq m$.

\smallskip
{\bf Claim.} If for some $A_1= \begin{bmatrix} A_{11} & A_{12} \\ A_{12}^* & A_{22}\end{bmatrix}\in\cS$ with $A_{11}\in S\R^{s\times s}$ we have
$\langle A_0, A_1\rangle=0$ and
\beq\label{eq:block2}
A_0 + \la A_1 = 
\begin{bmatrix} I_s+ \la A_{11} & \la A_{12} \\ \la A_{12}^* & \la A_{22}
\end{bmatrix}
\succeq 0 \quad\text{ for all } \quad \la\in\R_{\geq0},
\eeq
then 
$A_1=0$.

\emph{Explanation.}
Since $\langle A_0,A_1\rangle=0$, 
$\tr(A_{11})=0$. This means that either $A_{11}=0$ or $A_{11}$ has both positive and negative eigenvalues.
In the latter case, fix an eigenvalue $\mu<0$ of $A_{11}$.
Then for every $\la\in\R$ with $\la>-\mu^{-1}>0$, 
we have that $I_s+\la A_{11} \not\succeq0$, contradicting \eqref{eq:block2}.
So $A_{11}=0$. If $s=m$ we are done. Hence assume $s<m$.

Now
\beq\label{eq:preschur}
A_0+ \la A_1 = \begin{bmatrix} I_s  & \la A_{12} \\ \la A_{12}^* & \la A_{22}
\end{bmatrix}
\succeq 0
\eeq
for all $\la\in\R_{\geq0}$.
Using Schur complements, \eqref{eq:preschur}
is equivalent to
\bes%\label{eq:schur}
\la A_{22} - \la^2 A_{12}^*A_{12} \succeq0.
\ees
Hence $A_{22} - \la A_{12}^*A_{12} \succeq0$ for all $\la\in\R_{\geq0}$. 
Equivalently,
$A_{12}=0$ and $A_{22}\succeq0$.
If $A_{22}\neq0$, then $0 \preceq A_0+A_1 \in\cS$, and
$$s=\rank(A_0) < \rank(A_0+A_1),$$ contradicting the maximality of the rank of $A_0$.
\eop

\smallskip
Take an arbitrary orthogonal basis $A_0,A_1,\ldots, A_n$
 of $\cS$ containing $A_0$, and let
$L:=A_0+X_1 A_1+\cdots + X_n A_n$. 
We claim that $S_L$ is bounded.

Assume otherwise. 
Then there exists a sequence $(x^{(k)})_k$
in $\R^n$
such
  that $\|x^{(k)}\|=1$ for all $k$, and an increasing sequence
  $t_k\in\R_{>0}$ tending to $\infty$ such that
  $L(t_k x^{(k)})\succeq0$.  By convexity this implies $t_kx^{(r)}\in S_L$ for all $r\ge k$.
Without loss of generality we assume the sequence
   $(x^{(k)})_k$ converges to a vector $x=(x_1,\ldots,x_n)\in\R^n$.
Clearly, $\|x\|=1$.
For any $t\in\R_{\ge0}$, $tx^{(k)}\to tx$, and for $k$ big enough,
$tx^{(k)}\in S_{L}$ by convexity.
So $x$ satisfies $L(t x)\succeq 0$ for
all $t\in\R_{\geq0}$.
In other words,
\[ A_0 + t (x_1 A_1 + \cdots + x_n A_n) \geq0
\]
for all $t\in\R_{\geq0}$. But now the claim implies
$x_1A_1+\cdots+x_nA_n=0$, contradicting the linear independence of the $A_j$.
\end{proof}

\begin{lem}%\label{lem:classic}
Suppose $\cS\subseteq S\R^{m\times m}$ is a linear subspace and
$\varphi:\cS\to\R$ is a positive map. Then $\varphi$ is completely
positive.
\end{lem}

\begin{proof}
This is well-known and easy, cf.~\cite[Proposition 3.8 or Theorem 3.9]{Pau}.
\end{proof}

As seen in Example \ref{ex:za} 
there does not exist a clean linear certificate for LMI matricial domination
(or, equivalently, complete positivity). 
However, 
our Theorem \ref{indsos} yields
a nonlinear algebraic certificate for LMI (matricial)
domination $S_{L} \subseteq S_{f}$ in the case $f$ is a 
size $1$ linear pencil, i.e., $f\in\R\cx_1$. 
Equivalently, Theorem \ref{indsos} can be used to describe
(completely) positive maps $\tau:\cS\to\R$, where 
$\cS\subseteq S\R^{m\times m}$ is a(ny) linear subspace.

Indeed, suppose $\cS$ contains a nonzero positive semidefinite matrix 
(otherwise $\tau$ is automatically completely positive)
and apply Proposition \ref{prop:makeMeBounded}
to obtain a basis $A_0,\ldots, A_n$ of $\cS$, and 
the 
 linear pencil
\[
L := A_0 + X_1 A_1 + \cdots + X_n A_n
\]
of size $m$ with bounded $S_L$.
Define
\[
f:= \tau(A_0) + X_1 \tau(A_1) + \cdots + X_n \tau(A_n) \in \R\cx_1.
\]
Then, by Theorem \ref{thm:tau} (here is where the boundedness
of $S_L$ is used),
$\tau$ is completely positive if and only if 
$\cD_L\subseteq\cD_f$. By Proposition \ref{prop:1pos}
the latter is equivalent to $S_L\subseteq S_f$, i.e.,
$f|_{S_L}\geq 0$, and this is a situation
completely characterized by Theorem \ref{indsos}.

\end{document}